\documentclass{amsart}
\usepackage[utf8]{inputenc}
\usepackage[english]{babel}

\usepackage[dvips]{graphics,epsfig}
\usepackage{times}
\usepackage{amsmath,amsthm,amssymb}
\usepackage{array}
\usepackage{bbm}
\usepackage{xcolor}
\usepackage{verbatim}
\usepackage{upgreek}
\usepackage[foot]{amsaddr}
\usepackage{caption}
\usepackage{circuitikz}

\usepackage{tikz-cd}

\usepackage{geometry}

\newtheorem{theorem}{Theorem}[section]

\newtheorem{lemma}[theorem]{Lemma}
\newtheorem{question}[theorem]{Question}
\newtheorem{cor}[theorem]{Corollary}
\newtheorem{prop}[theorem]{Proposition}
\newtheorem{observation}[theorem]{Observation}
\newtheorem{fact}[theorem]{Fact}
\newtheorem{claim}[theorem]{Claim}

\theoremstyle{definition}
\newtheorem{definition}[theorem]{Definition}
\newtheorem{notation}[theorem]{Notation}
\newtheorem{example}[theorem]{Example}

\theoremstyle{remark}
\newtheorem{remark}[theorem]{Remark}

\def\Ind{\setbox0=\hbox{$x$}\kern\wd0\hbox to 0pt{\hss$\mid$\hss} \lower.9\ht0\hbox to 0pt{\hss$\smile$\hss}\kern\wd0}
\def\Notind{\setbox0=\hbox{$x$}\kern\wd0\hbox to 0pt{\mathchardef \nn=12854\hss$\nn$\kern1.4\wd0\hss}\hbox to 0pt{\hss$\mid$\hss}\lower.9\ht0 \hbox to 0pt{\hss$\smile$\hss}\kern\wd0}

\numberwithin{equation}{section}

\title{Approximations of the strict order property}
\author{ Scott Mutchnik
}
\address{Department of Mathematics, Statistics, and Computer Science, University of Illinois at Chicago}
\email{mutchnik@uic.edu}
\thanks{This work was supported by the National Science Foundation under Grant No. DMS-2303034, and by the European Union’s Horizon 2020 research and innovation programme under the Marie Sklodowska-Curie grant agreement No 101034255.}

\begin{document}

\begin{abstract}
We extend the original family of properties $\mathrm{NSOP}_{n}$ for $n \geq 3$ an integer, in Shelah's classical hierarchy of theories in first-order logic, to a family of properties $\mathrm{NSOP}_{r}$ for real numbers $r \geq 3$. Observing that the definition of the original properties $\mathrm{NSOP}_{n}$ for $n \geq 3$ can be restated so that it extends to the case where $n$ is replaced with any real number $r \geq 3$, we define a theory to be $\mathrm{NSOP}_{r}$ if there is no definable relation $R(x,y)$ with a sequence $\{a_{i}\}_{i < \omega}$ such that $\models R(a_{i}, a_{j})$ for $i <j$, but without any set $\{a_{\theta}\}_{\theta \in S^{1}}$ such that $\models R(a_{\theta}, a_{\psi})$ for all $\theta, \psi \in S^{1}$ with $\psi$ lying at most $\frac{2\pi}{r}$ radians counterclockwise from $\theta$. We give equivalent characterizations of these properties that demonstrate the robustness of their definition, and then, to make more tractable the question of whether the real-valued $\mathrm{NSOP}_{r}$ hierarchy is actually distinct from the integer-valued $\mathrm{NSOP}_{n}$ hierarchy, translate the properties $\mathrm{NSOP}_{r}$ from the language of first-order logic to the language of hereditary classes.

Motivated by our extension of the original integer-valued $\mathrm{NSOP}_{n}$ hierarchy to the real-valued $\mathrm{NSOP}_{r}$ hierarchy, and the translation of these hierarchies into the hereditary class setting, we obtain a new real-valued quantity of independent combinatorial interest, $\mathfrak{o}(\mathcal{H})$, associated with any hereditary class $\mathcal{H}$. We define $\mathfrak{o}(\mathcal{H})$ to be the supremum of the real values $r$ such that there exists some $AB \in \mathcal{H}$ with $\{A_{i}\}_{i< \omega} \in \mathcal{H}$ where $A_{i}A_{j} \equiv AB$ for $i < j$, but with no $\{ A_{\theta}\}_{\theta \in S^{1}} \in \mathcal{H}$ with $A_{\theta}A_{\psi} \equiv AB$ for all $\theta, \psi \in S^{1}$ for which $\psi$ lies at most $\frac{2\pi}{r}$ radians counterclockwise from $\theta$, where $X \equiv Y$ denotes equivalence of enumerated structures up to isomorphism. We show that, when $\mathcal{H}$ is defined by a finite family of forbidden weakly embedded substructures, $\mathfrak{o}(\mathcal{H})$ is an integer.

While Malliaris implicitly showed that the properties $\mathrm{NSOP}_{n}$ are equivalent to closure under helix-shaped covering maps or \textit{helix maps} between graphs, both our observation that the properties $\mathrm{NSOP}_{n}$ can be restated so that $n$ can be replaced with any real number at least 3, and our result that $\mathfrak{o}(\mathcal{H})$ is an integer when $\mathcal{H}$ is a hereditary class defined by a finite family of forbidden weakly embedded substructures, are even exhibited by a special class of helix maps, the \textit{interval helix maps}. These are helix maps that respect the direction of edges, and whose regions are disjoint unions of linearly ordered sets without any edges between them. Toward showing the conjectural claim that $\mathfrak{o}(\mathcal{H})$ is not an integer in general, and therefore that the real-valued $\mathrm{NSOP}_{r}$ hierarchy is distinct from the integer-valued $\mathrm{NSOP}_{n}$ hierarchy at the level of hereditary classes, we show that the statement that $\mathfrak{o}(\mathcal{H})$ is an integer in general cannot be exhibited by interval helix maps.

\end{abstract}

\maketitle

\section{Introduction}

\textit{Classification theory}, an area of mathematical logic at the core of model theory, starts with the question of determining the number of models of a given size of a logical theory. Morley, in his proof of the categoricity theorem in \cite{M65}, defines an ordinal-valued rank on definable sets, in order to classify first-order theories with exactly one model of size $\kappa$, where $\kappa$ is an uncountable cardinal. In his work in \cite{Sh90}, widely credited for introducing classification theory as a subject, Shelah extends Morley's results to develop a sweeping classification of first-order theories in terms of their spectrum of numbers of models of different cardinalities; for uncountable models, Hart, Hrushovski and Laskowski complete this classification in \cite{HHL00}. Shelah proves his results by dividing the first-order theories up into ``classifiable" theories and ``unclassifiable" theories, and then classifying the classifiable theories by developing a structure theory for models of \textit{stable} theories. Stability, a criterion for classifiability that can be defined both in terms of the \textit{order property} for individual formulas and in terms of a counting condition on types, is now well-established as a cornerstone property of theories, with instability indicating some level of additional complexity.

The central project of classification theory today, and one of the central projects of model theory as a whole, is to classify the more complex, unstable theories. In the course of his work on the number of models in \cite{Sh90}, Shelah initiates this by defining a hierarchy of unstable theories, consisting of properties that, just like stability, can be defined in terms of the combinatorics of individual formulas. Within Shelah's hierarchy, the class of $\mathrm{NSOP}$ theories, theories with the \textit{negation of the strict order property}, allows for an especially high degree of complexity, and understanding this class is one of the most troubling open-ended problems in the classification of unstable theories. A theory is $\mathrm{NSOP}$ if it does not interpret a partially ordered set with an infinite (ascending or descending) chain, and otherwise has the \textit{strict order property}, or $\mathrm{SOP}$. Because an order has no directed cycles of any size, one way to make understanding the strict order property more tractable is to consider the following family of approximations of the strict order property, which Shelah introduces in \cite{She95}:

\begin{definition}\label{original hierarchy} (Shelah, \cite{She95}.)
Let $n \geq 3$ be an integer. A theory $T$ is $\mathrm{NSOP}_{n}$ (that is, does not have the \emph{n-strict order property}) if there is no definable relation $R(x, y)$ such that

\begin{itemize}
    \item there exists a sequence $\{a_{i}\}_{i < \omega}$ such that $\models R(a_{i}, a_{j})$ for $i <j$, but

    \item there does not exist a directed $n$-cycle for $R(x, y)$; i.e., there are no $\{a_{i}\}^{n-1}_{i = 0}$ such that $\models R(a_{i}, a_{(i+1)\: \mathrm{mod}\: n}) $ for $0 \leq i \leq n-1$.
\end{itemize}
 Otherwise, it has $\mathrm{SOP}_{n}$.
\end{definition}

Later, Džamonja and Shelah (\cite{DS04}) introduce the properties $\mathrm{NSOP}_{1}$ and $\mathrm{NSOP}_{2}$, which are defined differently; see the introduction to \cite{SOPEXP} for an overview of our current understanding of $\mathrm{NSOP}_{n}$ theories for $n \geq 1$.\footnote{See also the interactive diagram designed by Conant at \cite{FD} for a visualization of the classical classification-theoretic hierarchy introduced by Shelah (\cite{Sh90}, \cite{She95}) and Džamonja and Shelah (\cite{DS04}), along with a few properties introduced more recently by other authors.} However, up to now, the properties $\mathrm{NSOP}_{n}$ have only been defined for integer values of $n$. In this article, we will be interested in the non-integer case.

As we will see, the definitions of the properties $\mathrm{NSOP}_{n}$ for integers $n \geq 3$ can be restated so that they remain valid for non-integer values of $n$, leading us to observe these properties could just as well have been defined for all real values greater than $3$, rather than just integer values. Specifically, we define the following family of properties of first-order theories, which extends Shelah's original $\mathrm{NSOP}_{n}$ hierarchy for integer values of $n$.

\begin{definition}\label{real-valued hierarchy}

Let $r \geq 3$ be a real number.\footnote{Though this definition will also make sense for, say, any real value $r > 2$, see remark \ref{ascending chain} for why we impose the restriction that $r \geq 3$.} A theory $T$ is $\mathrm{NSOP}_{r}$ if there is no definable relation $R(x, y)$ such that:

\begin{itemize}
    \item there exists a sequence $\{a_{i}\}_{i < \omega}$ such that $\models R(a_{i}, a_{j})$ for $i <j$, but

    \item there does not exist a set $\{a_{\theta}\}_{\theta \in S^{1}}$, indexed by the unit circle $S^{1}$, such that $\models R(a_{\theta}, a_{\psi})$ for all $\theta, \psi \in S^{1}$ with $\psi$ lying at most $\frac{2\pi}{r}$ radians counterclockwise from $\theta$ (see figure \ref{unit circle configuration} below.)
\end{itemize}

\begin{figure}[hbt!]
    
    \includegraphics[width=0.5\linewidth]{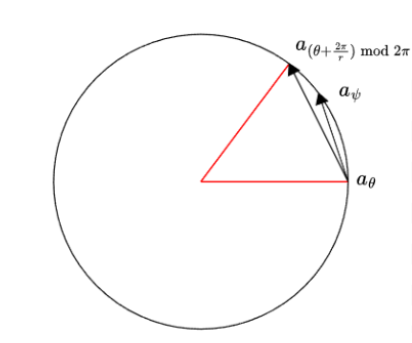}
    \caption{The forbidden set $\{a_{\theta}\}_{\theta \in S^{1}}$ in the definition of the properties $\mathrm{NSOP}_{r}$ for $r \geq 3$ a real number, Definition \ref{real-valued hierarchy}. Here angles are in radians, and the arrows denote the definable relation $R(x, y)$.}

    \label{unit circle configuration}
    
\end{figure}

 Otherwise it has $\mathrm{SOP}_{r}$.
    
\end{definition}

The fact that this definition specializes to the original properties $\mathrm{NSOP}_{n}$ in the case that $r = n \geq 3$ is an integer also offers new insight into the original properties $\mathrm{NSOP}_{n}$ themselves. According to the original definition (Definition \ref{original hierarchy}), in any $\mathrm{NSOP}_{n}$ theory, any definable relation $R(x, y)$ with a sequence $\{a_{i}\}_{i < \omega}$ such that $\models R(a_{i}, a_{j})$ for $i <j$ has a directed $n$-cycle. According to this equivalent definition with $r=n$, any such relation $R(x, y)$ has more than just a directed $n$-cycle: for example, $R(x,y)$ has infinitely many directed $n$-cycles, such as $\{a_{\frac{2k\pi}{n} \mathrm{\: rad \:}+ \psi}\}^{n-1}_{k=0}$ for fixed values of $\psi \in S^{1}$, which all interact with each other according to the hypothesis on $\{a_{\theta}\}_{\theta \in S^{1}}$.

This new definition leads us to the question of whether we have \textit{properly} extended Shelah's original $\mathrm{NSOP}_{n}$ hierarchy. That is, is there some $r \geq 3$ such that the class of $\mathrm{NSOP}_{r}$ theories is distinct from the class of $\mathrm{NSOP}_{n}$ theories for all integers $n \geq 3$? This question remains open, and we will consider it in Question \ref{distinctness} below. To make this question more tractable, we can also ask it at the level of combinatorics, rather than at the level of first-order logic. It is this combinatorial version of the question of whether the real-valued $\mathrm{NSOP}_{r}$ hierarchy is distinct from the integer-valued $\mathrm{NSOP}_{n}$ hierarchy that will interest us in this article.

Namely, we can ask the question of whether the two hierarchies are distinct at the \textit{quantifier-free} level, instead of at the level of full first-order logic (as in Remark \ref{hereditary classes and quantifier-free formulas} below). But equivalently, we can ask it in the setting of hereditary classes, straightforwardly translating the properties $\mathrm{NSOP}_{r}$ for $r \geq 3$ from the first-order setting to the hereditary class setting (as in Definition \ref{real-valued hierarchy, hereditary class version} below). Hereditary classes have previously been of interest to combinatorialists, and intrinsically combinatorial problems in the setting of hereditary classes have seen many connections to model theory: for example, Malliaris and Shelah (\cite{MS14}) show that stable graphs satisfy the conclusion of the Erdős-Hajnal conjecture on hereditary classes from combinatorics, as observed explicitly in \cite{CS16}, and Malliaris and Pillay further develop the connections to stability theory in \cite{MP15}. For further examples, see Laskowski and Terry (\cite{LC22}) for applications of model theory to speeds of hereditary classes, Terry and Wolf \cite{TW25} for applications of higher-arity stability, and Malliaris and Coregliano (\cite{CM22}, \cite{CM24}) for applications of stability to the Erdős-Hajnal problem via the theory of graphons. Translating the properties $\mathrm{SOP}_{r}$ into the hereditary class setting gives us a new quantity of intrinsic combinatorial interest associated with a hereditary class $\mathcal{H}$, the real number $\mathfrak{o}(\mathcal{H})$ defined in Definition \ref{real value of a hereditary class} below.

We will see that posing the problem of distinctness of the hierarchies in the language of hereditary classes allows us to make use of new combinatorial assumptions, such as being defined by a finite family of forbidden weakly embedded substructures (Definition \ref{defined by a (finite) family of forbidden weakly embedded substructures}). Our first main theorem, Theorem \ref{integerality}, will be to show that when a hereditary class $\mathcal{H}$ is defined by a finite family of forbidden weakly embedded substructures, $\mathfrak{o}(\mathcal{H})$ is an integer. As one should expect in any combinatorial analysis of the properties $\mathrm{NSOP}_{n}$, we will prove this using ``helix maps" (Definition \ref{helix map}), a kind of covering space of a graph implicitly introduced by Shelah in \cite{She95} in his proofs of the properties $\mathrm{SOP}_{n}$ in some examples of first-order theories, and by Malliaris in \cite{Mal10b} in her work on edge distribution in $\mathrm{NSOP}_{3}$ theories. The key insight will be to understand these helix covers not just as graphs, but as morphisms in a category of finite graphs, isolating an abstract property of helix maps (Lemma \ref{cycle removal}) that will allow us to remove induced $m$-cycles for $m \geq n$ from graphs omitted by any $\mathrm{NSOP}_{n}$ hereditary class. By repeatedly applying this abstract property to (some finite part of) the omitted configuration $\{a_{\theta}\}_{\theta \in S^{1}}$ in Figure \ref{unit circle configuration}, with a new stage for each cycle length, we eventually obtain an omitted configuration whose smallest cycles are arbitrarily large. By the assumptions on the hereditary class, we then use an argument from \cite{LZ17} to obtain an omitted cyclefree configuration, which cannot exist in the presence of $\{a_{i}\}_{i< \omega}$ with $\models R(a_{i}, a_{j})$ for $i < j$.

Our next step will be to consider the question of whether $\mathfrak{o}(\mathcal{H})$ is an integer in general (Question \ref{hereditary class distinctness}), which, up to the case where $\mathcal{H}$ is $\mathrm{NSOP}_{n}$ for $n$ an integer but has $\mathrm{SOP}_{r}$ for $3 \leq r < n$, will be the same as asking whether the real-valued and integer-valued hierarchies are distinct at the level of hereditary classes. If $\mathfrak{o}(\mathcal{H})$ can have non-integer values for some hereditary class $\mathcal{H}$, this would be of interest for two reasons. First, showing the real-valued $\mathrm{NSOP}_{r}$ and integer-valued $\mathrm{NSOP}_{n}$ hierarchies are distinct \textit{at the level of hereditary classes} would be relevant to the question of whether the real-valued $\mathrm{NSOP}_{r}$ hierarchy actually introduces new properties \textit{of first-order theories}. Second, if $\mathfrak{o}(\mathcal{H})$ is not always an integer in general, that would stand in contrast to our result that $\mathfrak{o}(\mathcal{H})$ \textit{is} always an integer when $\mathcal{H}$ is defined by a finite family of forbidden weakly embedded substructures. Towards showing that $\mathfrak{o}(\mathcal{H})$ is not an integer, in Theorem \ref{interval helix non-integrality} we prove that a certain class of helix maps that is powerful enough for showing $\mathfrak{o}(\mathcal{H})$ is an integer in the case where $\mathcal{H}$ is defined by a finite family of forbidden weakly embedded substructures, and is \textit{also} powerful enough for showing that Shelah's original integer-valued $\mathrm{NSOP}_{n}$ hierarchy can be restated to extend to the real-valued $\mathrm{NSOP}_{r}$ hierarchy defined above, is \textit{not} powerful enough to show that $\mathfrak{o}(\mathcal{H})$ is an integer in general.\footnote{Technically, it will be more precise to say that we will show that this class of helix maps is not powerful enough to show that the real-valued $\mathrm{SOP}_{r}$ hierarchy and integer-valued $\mathrm{NSOP}_{n}$ hierarchy are not distinct at the level of hereditary classes, but our proof should straightforwardly extend to the statement in terms of $\mathfrak{o}(\mathcal{H})$ by Remark \ref{real-valued helix maps} below.} This class of helix maps will be the \textit{interval helix maps} (Definition \ref{special helices}), the class of helix maps, respecting the direction of edges, whose components are disjoint unions of linearly ordered sets with no edges between them. The question of whether $\mathfrak{o}(\mathcal{H})$ is an integer, or of whether the real-valued and integer-valued hierarchies are distinct for hereditary classes, reduces to the case where $\mathcal{H}$ is a hereditary class of graphs closed under weak embeddings (Remark \ref{reduction to directed graphs}), and Malliaris implicitly showed in \cite{Mal10b} that such a hereditary class is $\mathrm{NSOP}_{n}$ if and only if it is closed under $n$-helix maps (Fact \ref{closure under helix maps}). So (taking into account the part of Remark \ref{reduction to directed graphs} about the graph relation itself exhibiting $\mathrm{SOP}_{r}$) a hereditary class of this kind will exhibit the distinctness of the real-valued and integer-valued hierarchies if its graph relation exhibits $\mathrm{SOP}_{r}$ for some real number $r > 3$, but is closed under $n$-helix maps for every integer $n > r$. Theorem \ref{interval helix non-integrality} says that there is at least a hereditary class whose graph relation exhibits $\mathrm{SOP}_{r}$ for some real number $r > 3$, but which is closed under $n$-interval helix maps for any integer $n > r$. We prove this by defining a combinatorial invariant, \textit{cyclic $n$-indecomposability} (Definition \ref{cyclically n-indecomposable}), which is satisfied by (a finite part of) the omitted configuration $\{a_{\theta}\}$ from Figure \ref{unit circle configuration}, and which is only satisfied by graphs containing cycles. Roughly, cyclic $n$-indecomposability says that a graph cannot be completely partitioned into $n$ regions so that the edges in between those regions have the pattern of a directed $n$-cycle. We will have to show that, for $n^{*} = \lceil r \rceil$ the next integer after $r$, this graph invariant is preserved under $n$-helix covers for integers $n >r $: any $n$-helix cover of an $n^{*}$-cyclically indecomposable graph must have an $n^{*}$-indecomposable subgraph. To obtain this subgraph, we construct a roughly ``parallelogram-shaped" set, as in Figure \ref{parallelogram} below.

The organization of this paper is as follows. In Section 2, we will show that Shelah's original properties $\mathrm{NSOP}_{n}$ (Definition \ref{original hierarchy}) for $n \geq 3$ an integer really do extend to the real-valued hierarchy defined in Definition \ref{real-valued hierarchy}, and show the equivalence of finitary and infinitary versions of the definition of the real-valued hierarchy, demonstrating the robustness of the definition of this hierarchy. These equivalent definitions are stated in Definition \ref{real-valued hierarchy, detailed definition}. Within Section 2, in Section 2.1 we translate the definition of the real-valued hierarchy from the first-order setting to the hereditary class setting (Definition \ref{real-valued hierarchy, hereditary class version}), and define the real-valued quantity $\mathfrak{o}(\mathcal{H})$ (Definition \ref{real value of a hereditary class}). We also pose the questions of whether the real-valued and integer-valued hierarchies are distinct at the level of hereditary classes, and of whether $\mathfrak{o}(\mathcal{H})$ is an integer (Question \ref{hereditary class distinctness}), define the concept of weak embeddings and related properties of hereditary classes (Definitions \ref{weak embedding}, \ref{class of structures omitting a weakly embedded substructure}, and \ref{defined by a (finite) family of forbidden weakly embedded substructures}), and show that these questions reduce to the case of a hereditary class defined by a family of forbidden weakly embedded substructures (Proposition \ref{reduction to the case of a hereditary class closed under weak embeddings}, which will be improved to the case of a hereditary class of \textit{directed graphs} in Remark \ref{reduction to directed graphs}). In Section 3, we prove that $\mathfrak{o}(\mathcal{H})$ is an integer in the case of a hereditary class $\mathcal{H}$ defined by a \textit{finite} family of forbidden weakly embedded substructures (Theorem \ref{integerality}), defining helix maps (Definition \ref{helix map}), and recounting Malliaris's proof that $\mathrm{NSOP}_{n}$ is equivalent to closure under $n$-helix maps (Fact \ref{closure under helix maps}), in the process. In Section 4, we define special kinds of helix maps, including interval helix maps (Definition \ref{special helices}), make precise how they are powerful enough to show that $\mathfrak{o}(\mathcal{H})$ is an integer in the case of Section 3 and to prove the restatement of Shelah's original properties, and prove our result on how they are not powerful enough to show that $\mathfrak{o}(\mathcal{H})$ is an integer in general (Theorem \ref{interval helix non-integrality}). Finally, within Section 4, in Section 4.1 we present some additional considerations relevant to the open problem of extending this result to more general classes of helix maps.

\section{Motivation and definitions}

We begin by showing that our newly defined $\mathrm{NSOP}_{r}$ hierarchy (Definition \ref{real-valued hierarchy}), for $r \geq 3$ a real number, actually extends Shelah's original $\mathrm{NSOP}_{n}$ hierarchy (Definition \ref{original hierarchy}), for $n \geq 3$ an integer. We also give some equivalent conditions for the properties $\mathrm{NSOP}_{r}$ for $r \geq 3$ a real number, showing that the definition of these properties is robust. These equivalent conditions for the properties $\mathrm{NSOP}_{r}$, for $r \geq 3$ a real number, will be summarized by a restatement of Definition \ref{real-valued hierarchy}, Definition \ref{real-valued hierarchy, detailed definition} below.

We start by proving the following observation. This observation will be supplanted by Lemma \ref{(N, k)-cycle comparison} below; however, we include it to isolate the main idea of why the definitions of the properties $\mathrm{NSOP}_{n}$ can be restated so that they make sense for non-integer values of $n$, especially in the last paragraph of the proof of this observation.

\begin{observation}\label{restatement of NSOP_n}
Let $n \geq 3$ be an integer. A theory $T$ is $\mathrm{NSOP}_{n}$ if and only if, for any definable relation $R(x, y)$ with a sequence $\{a_{i}\}_{i < \omega}$ such that $\models R(a_{i}, a_{j})$ for $i <j$, and for all $N < \omega$, there exist $b_{0}, \ldots, b_{N -1} $ such that, for all integers $i, j$ such that $0 \leq i <N$ and $1 \leq j \leq \frac{N}{n}$,  $\models R(b_{i}, b_{i+j \mathrm{\: mod \:} N})$.

\end{observation}

\begin{proof}
    The ``if" direction is immediate; just set $N = n$.

    For the ``only if" direction, let $T$ be $\mathrm{NSOP}_{n}$, let $R(x, y)$ be any relation with a sequence $\{a_{i}\}_{i < \omega}$ such that $\models R(a_{i}, a_{j})$ for $i <j$, and let $N < \omega$. We show that there exist $b_{0}, \ldots, b_{N -1} $ such that, for all integers $i, j$ such that $0 \leq i <N$ and $1 \leq j \leq \frac{N}{n}$,  $\models R(b_{i}, b_{i+j \mathrm{\: mod \:} N})$.

    Let us first reduce to the case where $N$ is a multiple of $n$. Assume we have shown that case; then it remains to find $b_{0}, \ldots, b_{N -1} $ as above for any $N$ that is not a multiple of $n$. Let $N'=kn$ be the least multiple of $n$ greater than $N$. Then there exist $b'_{0}, \ldots, b'_{N' -1} $ such that, for all integers $i, j$ such that $0 \leq i <N'$ and $1 \leq j \leq k $,  $\models R(b_{i}, b_{i+j \mathrm{\: mod \:} N'})$. If we can find $b_{0}, \ldots, b_{N -1} $ such that for all integers $i, j$ such that $0 \leq i <N$ and $1 \leq j \leq k-1 $,  $\models R(b_{i}, b_{i+j \mathrm{\: mod \:} N})$, these will be as desired. Let $l = N' - N$; since $l < n$, we can choose $0 \leq \hat{i}_{1} < \ldots < \hat{i}_{l}\leq N'-1$ such that, for $1 \leq j < l$, $\hat{i}_{j+1} - \hat{i}_{j} $ is not equal $\mathrm{mod} \: N'$ to one of $0, 1, \ldots k-1$, and $\hat{i}_{1} - \hat{i}_{l} $ is also not equal $\mathrm{mod} \: N'$ to one of $0, 1, \ldots k-1$. For $0 \leq j \leq N-1$, define $i_{j}$ to be the $j$th largest element of $\{0, \ldots, N'-1\} \backslash \{\hat{i}_{1} , \ldots , \hat{i}_{l}\}$, and define $b_{j} := b'_{i_{j}}$. Then it is in fact true that for all integers $i, j$ such that $0 \leq i <N$ and $1 \leq j \leq k-1 $,  $\models R(b_{i}, b_{i+j \mathrm{\: mod \:} N})$, by the assumption on $b'_{0}, \ldots, b'_{N' -1} $ and the fact that, for all $i \in \{0, \ldots, N'-1\} \backslash \{\hat{i}_{1} , \ldots , \hat{i}_{l}\}$, at most one of the values $(i+1 \: \mathrm{mod} \: N'), \ldots, (i + k \: \mathrm{mod} \: N') $ belongs to $\{\hat{i}_{1} , \ldots , \hat{i}_{l}\}$.

    Now we show the case where $N = kn$. Define $\{\tilde{a}_{i}\}_{i < \omega} := \{a_{ik}a_{ik+1} \ldots a_{ik+(k-1)} \}_{i < \omega} $. For variables $\tilde{x}_{1} := x_{0}\ldots x_{k-1} $, $\tilde{x}_{2} := x_{k} \ldots x_{2k-1}$, define $\tilde{R}(\tilde{x}_{1}, \tilde{x}_{2}) := \bigwedge_{0 \leq i < j \leq 2k -1} R(x_{i}, x_{j})$. Then $\{\tilde{a}_{i}\}_{i < \omega}$ satisfies the property that $\models \tilde{R}(\tilde{a}_{i}, \tilde{a}_{j})$ for $i < j$. So $\tilde{R}(\tilde{x}_{1}, \tilde{x}_{2})$ has a directed $n$-cycle, $\{\tilde{b}_{i}\}^{n}_{i = 0}$. Write $\tilde{b}_{i}$ as $b_{ik}b_{ik+1} \ldots b_{ik+(k-1)}$. Then $b_{0}, \ldots, b_{N -1} $ will be as desired.
\end{proof}

Observe that the equivalent condition for $\mathrm{NSOP}_{n}$ stated in Observation \ref{restatement of NSOP_n} really does make sense for non-integer values of $n$, suggesting we can in fact extend the definitions of the properties $\mathrm{NSOP}_{n}$ for $n \geq 3$ an integer to definitions of properties $\mathrm{NSOP}_{r}$ for $r \geq 3$ a real number, as we did in Definition \ref{real-valued hierarchy} from the introduction. 

We prove the following two lemmas to justify the below equivalent definitions of $\mathrm{NSOP}_{r}$ for $r \geq 3$ a real number (Definition \ref{real-valued hierarchy, detailed definition}, giving equivalent conditions for Definition \ref{real-valued hierarchy}.) We introduce the following notation only as a notational convenience to use in our proofs; in our statements of definitions and main theorems, we will not use this notation, in order to keep those definitions and theorems explicit.

\begin{notation}\label{(N, k)-cycle}

Let $R(x, y)$ be a definable relation, and let $N$ and $k$ be positive integers. An \textit{$(N, k)$-cycle} for $R(x, y)$ is a set  $b_{0}, \ldots, b_{N -1} $ such that, for all integers $i, j$ such that $0 \leq i <N$ and $1 \leq j \leq k $,  $\models R(b_{i}, b_{i+j \mathrm{\: mod \:} N})$.
\end{notation}

\begin{lemma}\label{(N, k)-cycle comparison}
    Let $N, k$ be positive integers. Say that a theory $T$ has property \emph{(*)}$_{N, k}$ if, for any definable relation $R(x, y)$ with a sequence $\{a_{i}\}_{i < \omega}$ such that $\models R(a_{i}, a_{j})$ for $i <j$, $R(x, y)$ has an $(N, k)$-cycle. Suppose $\frac{N}{k} \leq \frac{N'}{k'}$. Then if $T$ has \emph{(*)}$_{N, k}$, $T$ has \emph{(*)}$_{N', k'}$ 
\end{lemma}

\begin{proof}
    We first observe that it suffices to prove the following claim:

    \begin{claim}\label{multiply numerator and denominator}
        For any positive integers $m$, $N$, $k$, $T$ has \emph{(*)}$_{N, k}$ if and only if $T$ has \emph{(*)}$_{mN, mk}$.
    \end{claim}

    By this claim, it would suffice to prove Lemma \ref{(N, k)-cycle comparison} in the case where $N = N'$. But this is just immediate from the definition of an $(N, k)$-cycle.

    \begin{proof}(of claim)
        ($\Leftarrow$) Let $b'_{0}, \ldots, b'_{mN-1}$ be an $(mN, mk)$-cycle. For $0 \leq i < N$, let $b_{i} := b'_{m(i+1) -1} $. Then $b_{0}, \ldots, b_{N-1}$ is an $(N, k)$-cycle.

        ($\Rightarrow$) Suppose that $R(x, y)$ has a sequence $\{a_{i}\}_{i < \omega}$ such that $\models R(a_{i}, a_{j})$ for $i <j$, and that $T$ has (*)$_{N, k}$. We show that $R(x, y)$ has an $(mN, mk)$-cycle.

        Define $\{\tilde{a}_{i}\}_{i < \omega} := \{a_{im}a_{im+1} \ldots a_{im+(m-1)} \}_{i < \omega} $. For variables $\tilde{x}_{1} := x_{0}\ldots x_{m-1} $, $\tilde{x}_{2} := x_{m} \ldots x_{2m-1}$, define $\tilde{R}(\tilde{x}_{1}, \tilde{x}_{2}) := \bigwedge_{0 \leq i < j \leq 2m -1} R(x_{i}, x_{j})$. Then $\{\tilde{a}_{i}\}_{i < \omega}$ satisfies the property that $\models \tilde{R}(\tilde{a}_{i}, \tilde{a}_{j})$ for $i < j$. So $\tilde{R}(\tilde{x}_{1}, \tilde{x}_{2})$ has a directed $(N, k)$-cycle, $\{\tilde{b}_{i}\}^{N}_{i = 0}$. Write $\tilde{b}_{i}$ as $b_{im}a_{im+1} \ldots b_{im+(m-1)}$. Then $b_{0}, \ldots, b_{mN -1} $ will be as desired.
    \end{proof}
\end{proof}

\begin{lemma}\label{finitary and infinitary conditions}
    Let $r \geq 3$. The following are equivalent for any theory $T$:

    (1) $T$ has \emph{(*)}$_{N, k}$ for all integers $N, k$ with $\frac{N}{k} \geq r$.

    (2) $T$ is $\mathrm{NSOP}_{r}$ (as in Definition \ref{real-valued hierarchy}.)
\end{lemma}

\begin{proof}
    (2 $\Rightarrow$ 1): It suffices to show that if $R(x, y)$ is a definable relation, $\frac{N}{k} \geq r$, and there exist $\{a_{\theta}\}_{\theta \in S^{1}}$ such that $\models R(a_{\theta}, a_{\psi})$ for all $\theta, \psi \in S^{1}$ with $\psi$ lying at most $\frac{2\pi}{r}$ radians counterclockwise from $\theta$, then $R(x, y)$ has a $(N, k)$-cycle. For $0 \leq i \leq N-1$, let $\theta_{i}$ be $\frac{2\pi i}{N}$ radians, and define $b_{i} := a_{\theta_{i}}$. Then $b_{0}, \ldots, b_{N-1}$ is an $(N, k)$-cycle.

    (1 $\Rightarrow$ 2): It suffices to show that if $R(x, y)$ has an $(N, k)$-cycle for every $r \leq \frac{N}{k}$, the set

    $$\{R(x_{\theta}, x_{\psi}): \theta, \psi \in S^{1}, \psi \mathrm{\: lies \: at \: most \:}\frac{2\pi}{r}\mathrm{\: radians \: counterclockwise \: from \:} \theta \}$$

    is consistent. By compactness, it suffices to show that any finite subset $S$ of this set is consistent.
    
    Let us first reduce to the case that, for all variables $x_{\theta}$ mentioned in the sentences in $S$, $\theta$ is a rational multiple of $2\pi$ radians.
    
    If $r$ is irrational, we first reduce to the subcase that, for any two angles $\theta_{1}$, $\theta_{2}$ in the set of angles $\theta$ where $x_{\theta}$ is mentioned in $S$, $\theta_{1}$ does not lie \textit{exactly} $\frac{2\pi}{r}$ radians counterclockwise from $\theta_{2}$. We can partition this set of angles into maximal progressions of the form $\theta, \theta + \frac{2\pi}{r} \mathrm{\: rad}, \ldots, \theta + \frac{2n\pi}{r} \mathrm{\: rad} $. Then for any $\theta_{1}$, $\theta_{2}$ in this set such that $\theta_{1}$ lies \textit{exactly} $\frac{2\pi}{r}$ many radians counterclockwise from $\theta_{2}$, $\theta_{1}$ and $\theta_{2}$ belong to the same element of this partition, and are successive terms of the progression. We may assume that at least one of the elements of the partition isn't a singleton, because otherwise we are done with the reduction to this subcase. Since $r$ is irrational, for $\theta, \theta + \frac{2\pi}{r} \mathrm{\: rad}, \ldots, \theta + \frac{2n\pi}{r} \mathrm{\: rad} $ a non-singleton element of the partition, this sequence actually enumerates a set of $n+1$ distinct angles, so for sufficiently small $ \varepsilon > 0$, replacing $\theta + \frac{2k\pi}{r} \mathrm{\: rad}  $ with $\theta  k(\frac{2\pi}{r} - \varepsilon) \mathrm{\: rad}$ within this element of the partition, while keeping angles the same outside of this element of this partition, gives us a one-to-one correspondence between the original set of angles and a new set of angles. Since the original set is finite, we may also choose $\varepsilon$ to be small enough that, when two angles $\theta_{1}$, $\theta_{2}$ from the original set are such that $\theta_{1}$ lies \textit{less than} $\frac{2\pi}{r}$ radians counterclockwise from $\theta_{2}$, then the same is true of their replacements $\theta'_{1}$ and $\theta'_{2}$, while when $\theta_{1}$, $\theta_{2}$ from the original set are such that $\theta_{1}$ \textit{does not} lie at most $\frac{2\pi}{r}$ radians counterclockwise from $\theta_{2}$, the same is true of their replacements $\theta'_{1}$ and $\theta'_{2}$. When $\theta_{1}$, $\theta_{2}$ from the original set are such that $\theta_{1}$ lies \textit{exactly} $\frac{2\pi}{r}$ radians counterclockwise from $\theta_{2}$, and are therefore successive terms in the same element of the partition, they are either successive terms in some other element of the partition besides the chosen progression $\theta, \theta + \frac{2\pi}{r} \mathrm{\: rad}, \ldots, \theta + \frac{2n\pi}{r} \mathrm{\: rad} $, so are unchanged, or are successive terms in $\theta, \theta + \frac{2\pi}{r} \mathrm{\: rad}, \ldots, \theta + \frac{2n\pi}{r} \mathrm{\: rad} $, so their replacements $\theta'_{1}$ and $\theta'_{2}$ are such that $\theta'_{1}$ lies \textit{less than} $\frac{2\pi}{r}$ radians counterclockwise from $\theta'_{2}$. To summarize, for any two angles $\theta_{1}$, $\theta_{2}$ in the original set, $\theta_{1}$ lies at most $\frac{2\pi}{r}$ radians counterclockwise from $\theta_{2}$ if and only if the same is true for their replacements $\theta'_{1}$, $\theta'_{2}$, while the new set has fewer instances of an angle lying \textit{exactly} $\frac{2\pi}{r}$ radians counterclockwise from $\theta_{2}$. By induction on the number of these instances (making the corresponding replacements of sentences in $S$) we have successfully reduced to the case that, for any two angles $\theta_{1}$, $\theta_{2}$ belonging to the set of all angles $\theta$ such that $x_{\theta}$ is mentioned in $S$, $\theta_{1}$ does not lie \textit{exactly} $\frac{2\pi}{r}$ radians counterclockwise from $\theta_{2}$.
    
    Now that we are in this subcase, for any angle $\theta$ in this set that is not a rational multiple of $2\pi$ radians, we may, by density of the rational numbers, choose an angle $\theta'$ that \textit{is} a rational multiple of $2\pi$ radians so that the one-to-one correspondence given by replacing $\theta$ with $\theta'$, while keeping all of the other angles in the set the same, preserves whether one angle lies less than $\frac{2\pi}{r}$ radians counterclockwise from the other, whether one angle lies exactly $\frac{2\pi}{r}$ radians counterclockwise from the other, or whether one angle does not lie at most $\frac{2\pi}{r}$ radians counterclockwise from the other. By induction on the number of angles that are not rational multiples of $2\pi$ radians, we have successfully reduced to the case that all angles $\theta$ such that $x_{\theta}$ is mentioned in $S$ are rational multiples of $2\pi$ radians.

    If $r$ is rational, then we can partition the set of angles $\theta$ such that $x_{\theta}$ is mentioned in the sentences in $S$ according to the equivalence relation of lying in the same $R^{1}$-coset of $S^{1}$, where $R^{1}$ is the group of rational multiples of $2\pi$ radians in $S^{1}$. Since $r$ is rational, if $\theta_{1}$ lies exactly $\frac{2\pi}{r}$ many radians counterclockwise from $\theta_{2}$, then $\theta_{1}$ and $\theta_{2}$ must be in the same class. So if one (and therefore all) of the angles in some class $E$ is not a rational multiple of $2\pi$ radians, by the density of the rationals we can find a one-to-one correspondence, given by translating all of $E$ by some small enough translation while keeping all of the other angles the same, that replaces all of the elements of $E$ with rational multiples of $2\pi$ radians while preserving whether one angle lies at most $\frac{2\pi}{r}$ radians counterclockwise from the other. Repeating this, we have successfully reduced to the case that all angles $\theta$ such that $x_{\theta}$ is mentioned in $S$ are rational multiples of $2\pi$ radians.

    But then, if all angles $\theta$ such that $x_{\theta}$ is mentioned in $S$ are rational multiples of $2\pi$ radians, we may assume that the set of angles $\theta$ such that $x_{\theta}$ is mentioned in $S$ is equal, for some integer $N$, to the set $\{\theta_{i}\}^{N-1}_{i = 0}$, where $\theta_{i}$ is $\frac{2i\pi}{N}$ radians. Let $k = \lfloor \frac{N}{r} \rfloor$; then for each $\theta_{i}$, $k$ is equal to the number of other $\theta_{j}$ lying at most $\frac{2\pi}{r}$ radians counterclockwise from $\theta$, and $\frac{N}{k} \geq r$. So by assumption, we can find an $(N, k)$-cycle $b_{0}, \ldots b_{N-1}$ for $R(x, y)$. Then, instantiating $b_{i}$ in the variable $x_{\theta_{i}}$, we realize $S$, so $S$ is consistent.
\end{proof}

We now present Definition \ref{real-valued hierarchy} again, along with some equivalent conditions. Note that, by these equivalent conditions and Observation \ref{restatement of NSOP_n}, this definition agrees with Shelah's original $\mathrm{NSOP}_{n}$ hierarchy for $n \geq 3$ a real number (Definition \ref{original hierarchy}).

\begin{definition}\label{real-valued hierarchy, detailed definition}
    Let $r\geq 3$ be a real number. A theory $T$ is $\mathrm{NSOP}_{r}$ if the following equivalent conditions hold: 

    (1) For every definable relation $R(x_{1}, x_{2})$ with a sequence $\{a_{i}\}_{i < \omega}$ such that $\models R(a_{i}, a_{j})$ for all $i < j$, for all integers $N < \omega$ there are $b_{0}, \ldots, b_{N -1} $ such that $\models R(b_{i}, b_{i+j \mathrm{\: mod \:} N})$ for $0 \leq i <N$, $1 \leq j \leq \frac{N}{r}$.

    (1$'$) For every definable relation $R(x_{1}, x_{2})$ with a sequence $\{a_{i}\}_{i < \omega}$ such that $\models R(a_{i}, a_{j})$ for all $i < j$, for all \textit{sufficiently large} integers $N < \omega$ there are $b_{0}, \ldots, b_{N -1} $ such that $\models R(b_{i}, b_{i+j \mathrm{\: mod \:} N})$ for $0 \leq i <N$, $1 \leq j \leq \frac{N}{r}$.

    (2) For every definable relation $R(x_{1}, x_{2})$ with a sequence $\{a_{i}\}_{i < \omega}$ such that $\models R(a_{i}, a_{j})$ for all $i < j$, there exists a set $\{a_{\theta}\}_{\theta \in S^{1}}$, indexed by the unit circle $S^{1}$, such that $\models R(a_{\theta}, a_{\psi})$ for all $\theta, \psi \in S^{1}$ with $\psi$ lying at most $\frac{2\pi}{r}$ radians counterclockwise from $\theta$.

 Otherwise it has $\mathrm{SOP}_{r}$.
\end{definition}

We first prove the equivalence (1) $\Leftrightarrow$ (2). Condition (1) says that $T$ has property (*)$_{N, k}$ (as in the statement of Lemma \ref{(N, k)-cycle comparison} for all integers $N$, $k$ with $k \leq \frac{N}{r}$, but $k \leq \frac{N}{r}$ is equivalent to $r \leq \frac{N}{k}$, so (1) is equivalent to (2) by Lemma \ref{finitary and infinitary conditions}.

Now we prove the equivalence (1) $\Leftrightarrow$ (1$'$), reflecting the robustness of the definition. By the previous paragraph, condition (1) says that $T$ has property (*)$_{N, k}$ (as in the statement of Lemma \ref{(N, k)-cycle comparison}) for all integers $N$, $k$ with $r \leq \frac{N}{k}$, while condition (2) says that $T$ has property (*)$_{N, k}$ (as in the statement of Lemma \ref{(N, k)-cycle comparison}) for all integers $N$, $k$ with $r \leq \frac{N}{k}$, as long as $N$ is sufficiently large. But the property (*)$_{N, k}$ is equivalent to (*)$_{N', k'}$ for some $N', k'$ with $N'$ sufficiently large, by Claim \ref{multiply numerator and denominator}.

\begin{remark}\label{rationals and irrationals}
    Rational and irrational values of $r$ have different statuses within Definition \ref{real-valued hierarchy, detailed definition}. When $r \geq 3$ is a rational number, the property $\mathrm{NSOP}_{r}$ is equivalent to

    (1$'_{\mathbb{Q}}$) For every definable relation $R(x_{1}, x_{2})$ with a sequence $\{a_{i}\}_{i < \omega}$ such that $\models R(a_{i}, a_{j})$ for all $i < j$, for some integer $N < \omega$ that is a multiple of $r$ there are $b_{0}, \ldots, b_{N -1} $ such that $\models R(b_{i}, b_{i+j \mathrm{\: mod \:} N})$ for $0 \leq i <N$, $1 \leq j \leq \frac{N}{r}$.

To see this, note that (1$'_{\mathbb{Q}}$) says that, for some $N$ such that $\frac{N}{r}$ is an integer, $T$ satisfies (*)$_{N, \frac{N}{r}}$. But then, by Lemma \ref{(N, k)-cycle comparison}, $T$ satisfies (*)$_{N', k'}$ whenever $r \leq \frac{N'}{k'}$, so (1$'_{\mathbb{Q}}$) is equivalent to (1).

It follows from this equivalence that, when $n, m$ are integers and $\frac{n}{m} \geq 3$, the property $\mathrm{NSOP}_{\frac{n}{m}}$ is equivalent to the condition that:

For every definable relation $R(x_{1}, x_{2})$ with a sequence $\{a_{i}\}_{i < \omega}$ such that $\models R(a_{i}, a_{j})$ for all $i < j$, there are $b_{0}, \ldots, b_{n -1} $ such that $\models R(b_{i}, b_{i+j \mathrm{\: mod \:} n})$ for $0 \leq i <n$, $1 \leq j \leq m$.

On the other hand, when $r \geq 3$ is an irrational number, the definition of $\mathrm{NSOP}_{r}$ exhibits additional robustness. It is equivalent to

(1$'_{\mathbb{R} \backslash \mathbb{Q}}$) For every definable relation $R(x_{1}, x_{2})$ with a sequence $\{a_{i}\}_{i < \omega}$ such that $\models R(a_{i}, a_{j})$ for all $i < j$, for \textit{arbitrarily large} integers $N < \omega$ there are $b_{0}, \ldots, b_{N -1} $ such that $\models R(b_{i}, b_{i+j \mathrm{\: mod \:} N})$ for $0 \leq i <N$, $1 \leq j \leq \frac{N}{r}$.

To see this, note that the condition (1$'_{\mathbb{R} \backslash \mathbb{Q}}$) says that, for arbitrarily large $N$, for any $k$ such that $r \leq \frac{N}{k}$, $T$ has (*)$_{N, k}$. But since $r$ is irrational, for any rational number $q \geq r$, we then may choose such $N, k$ such that $\frac{N}{k} < q$. By Lemma \ref{(N, k)-cycle comparison}, (1$'_{\mathbb{R} \backslash \mathbb{Q}}$) then implies that $T$ satisfies (*)$_{N', k'}$ whenever $r \leq \frac{N'}{k'}$, so is equivalent to (1).

Moreover, the property $\mathrm{NSOP}_{r}$ for $r \geq 3$ irrational can be defined in terms of the properties $\mathrm{NSOP}_{q}$ for $q \geq 3$ rational. Specifically, for $r \geq 3$ irrational,

$$\mathrm{SOP}_{r} = \bigcup_{r < q \in \mathbb{Q}}\mathrm{SOP}_{q}$$.

This is because, if $T$ has $\mathrm{SOP}_{r}$, it must fail (*)$_{N, k}$ for some integers $N, k$ with $r < \frac{N}{k}$, so must have $\mathrm{SOP}_{\frac{N}{k}}$.

By this (i.e., $\mathrm{NSOP}_{r} = \bigcap_{r < q \in \mathbb{Q}}\mathrm{NSOP}_{q}$) and compactness, when $r \geq 3$ is irrational, $\mathrm{NSOP}_{r}$ is also equivalent to

(2$_{\mathbb{R} \backslash \mathbb{Q}}$) For every definable relation $R(x_{1}, x_{2})$ with a sequence $\{a_{i}\}_{i < \omega}$ such that $\models R(a_{i}, a_{j})$ for all $i < j$, there exists a set $\{a_{\theta}\}_{\theta \in S^{1}}$, indexed by the unit circle $S^{1}$, such that $\models R(a_{\theta}, a_{\psi})$ for all $\theta, \psi \in S^{1}$ with $\psi$ lying \textit{less than} $\frac{2\pi}{r}$ radians counterclockwise from $\theta$.

\end{remark}

\begin{remark}\label{ascending chain}
    The original $\mathrm{NSOP}_{n}$ hierarchy also includes the properties $\mathrm{NSOP}_{1}$ and $\mathrm{NSOP}_{2}$, defined by Shelah in \cite{Sh90}, \cite{Sh99} and by Džamonja and Shelah in \cite{DS04}. Together, these properties form an ascending chain:

    $$\mathrm{NSOP}_{1} = \mathrm{NSOP}_{2} \subseteq \mathrm{NSOP}_{3} \subsetneq \ldots \subsetneq \mathrm{NSOP}_{n} \subsetneq \mathrm{NSOP}_{n+1} \subsetneq \ldots$$

    It is clear from the definition that the properties $\mathrm{NSOP}_{r}$, for $r, s \geq 3$ a real number, extend this hierarchy: for real numbers $r, s \geq 3$, $\mathrm{NSOP}_{r} \subseteq \mathrm{NSOP}_{s}$ whenever $r \geq s$. This is not known to be a strict inclusion when s and r lie between the same pair of successive integers: see Question \ref{distinctness} below.

    The reason for the assumption $r \geq 3$ in the definition is that, when, say, $2 < \frac{N}{k} < 3$, it is not known that $\mathrm{NSOP}_{2}$ implies (*)$_{N, k}$; therefore, extending Definition \ref{real-valued hierarchy, detailed definition} to values less than $3$ may not preserve the fact that this hierarchy is an ascending chain.

\end{remark}

As stated in the introduction, it is open whether, by extending Shelah's original $\mathrm{NSOP}_{n}$ hierarchy from integer values of $n \geq 3$ to real values greater than $3$, we have actually properly extended the hierarchy:

\begin{question}\label{distinctness}

Is the $\mathrm{NSOP}_{r}$ hierarchy for $r \geq 3$ a real number distinct from the $\mathrm{NSOP}_{n}$ hierarchy for $n \geq 3$ an integer? That is, is there some real number $r \geq 3$ such that the class of $\mathrm{NSOP}_{r}$ theories is distinct from the class of $\mathrm{NSOP}_{n}$ theories for all $n \geq 3$?
\end{question}


\subsection{Hereditary classes}\label{hereditary class subsection}

In this section, we extend the preceding discussion to hereditary classes of relational structures, obtaining a new, real-valued combinatorial quantity associated with any hereditary class. From the point of view of logic, investigating the properties $\mathrm{NSOP}_{r}$ for $r \geq 3$ a real number in the setting of hereditary classes is the same (see Remark \ref{hereditary classes and quantifier-free formulas}) as investigating the properties $\mathrm{NSOP}_{r}$ for $r \geq 3$ a real number at the \textit{quantifier-free} level in any theory. Even at this syntactic level, the distinctness of the properties $\mathrm{NSOP}_{r}$ for $r \geq 3$ a real number and the properties $\mathrm{NSOP}_{n}$ for $n \geq 3$ a natural number remains open. However, stating this problem in the language of hereditary classes will allow us to make use of combinatorial assumptions that are specific to the hereditary class setting, making the problem more tractable (as we will see in the next two sections) and leading to results of independent combinatorial interest.

In what follows, we will consider hereditary classes in a finite relational language; the results of this paper can be straightforwardly generalized to infinite languages.

\begin{definition}\label{hereditary class}

Let $\mathcal{L}$ be a finite relational language. A hereditary class of $\mathcal{L}$-structures is a class $\mathcal{H}$ of finite $\mathcal{L}$-structures, considered up to isomorphism, such that if $B \in \mathcal{H}$, and $A \subseteq B$ is an (induced) substructure of $B$, then $A \in \mathcal{H}$.
    
\end{definition}

In the setting of a hereditary class $\mathcal{H}$ of $\mathcal{L}$-structures, we adapt the standard model-theoretic notation to $\mathcal{L}$-structures: when referring to a $\mathcal{L}$-structure $A$, $A$ may be understood as an ordered tuple (so $A \in \mathcal{H}$ means that its underlying $\mathcal{L}$-structure belongs to $\mathcal{H}$, $B \subseteq A$ will be enumerated so that the resulting ordering agrees with the ordering given by the enumeration of $A$, for a one-to one map $\iota: A \hookrightarrow B$, the enumeration of $B$ is compatible according to $\iota$ with the enumeration on $A$, etc.); this will be clear from context. Moreover, we may write $AB$ to refer to an $\mathcal{L}$-structure that is the union of subtructures $A$ and $B$; we may also speak of a sequence $\{A_{i}\}_{i \in I}$ interchangeably with the union of its terms, again in a way that will be clear from context. Crucially, for $\mathcal{L}$-substructures of the form $AC$ and $BC$ where $C$ is a common substructure, we write $A \equiv_{C} B$ to mean that the enumerations of $AC$ and $BC$ give an isomorphism of $\mathcal{L}$-structures; that is, if $AC$ and $BC$ are finite $\mathcal{L}$-structures enumerated as $AC = (a_{1}, \ldots a_{n} ) $, $BC = (b_{1}, \ldots b_{n}) $ with $a_{i} \in C$ if and only if $b_{i} \in C$ and $a_{i} = b_{i}$, then for any $n$-ary relation symbol $R \in \mathcal{L}$ (including equality), $(a_{i_{1}}, \ldots a_{i_{n}}) \in R(AC)$ if and only if $(b_{i_{1}}, \ldots b_{i_{n}}) \in R(AC)$. In the language of model theory, $A \equiv_{C} B$ means that the ordered tuples $AC$ and $BC$ have the same \textit{quantifier-free} $\mathcal{L}$-type.

Finally, abusing notation, for any hereditary class $\mathcal{H}$ of $\mathcal{L}$-structures and any \textit{infinite} $\mathcal{L}$-structure $A$, we use $A \in \mathcal{H}$ to denote that, for every \textit{finite} substructure $A_{0} \subset A$, $A_{0} \in \mathcal{H}$.

We will define the properties $\mathrm{NSOP}_{r}$ for $r \geq 3$ a real number, for any hereditary class $\mathcal{H}$, by straightforwardly translating Definition \ref{real-valued hierarchy, detailed definition} from the setting of first-order logic to the setting of hereditary classes. This will give us a real-valued combinatorial quantity $\mathfrak{o}(\mathcal{H})$ associated with any hereditary class. While $\mathfrak{o}(\mathcal{H})$ will be defined in terms of the properties $\mathrm{NSOP}_{r}$ for $r \geq 3$ a real number, we present its definition explicitly first, to isolate it as a quantity of independent combinatorial interest.

\begin{definition}\label{real value of a hereditary class}
Let $\mathcal{H}$ be a hereditary class of $\mathcal{L}$-structures.

        Let $\mathcal{H}$ be a hereditary class. Then $\mathfrak{o}(\mathcal{H})$ is the (possibly infinite) supremum of the real values $r \geq 3$ such that there exists some finite $\mathcal{L}$-structure of the form $AB$ such that

    \begin{itemize}
        \item There exists a sequence $\{ A_{i}\}_{i < \omega} \in \mathcal{H}$ (more explicitly, an infinite $\overline{A} \in \mathcal{H}$, and $\{ A_{i}\}_{i < \omega}$ with the tuples $A_{i} \subseteq \overline{A}$), such that $A_{i}A_{j} \equiv AB$ for all $i < j$.
        \item There does not exist a set $\{ A_{\theta}\}_{\theta \in S^{1}} \in \mathcal{H}$ (more explicitly, an infinite $\overline{A} \in \mathcal{H}$, and $\{ A_{\theta}\}_{\theta \in S^{1}}$ with the tuples $A_{\theta} \subseteq \overline{A}$) indexed by the unit circle $S^{1}$, with $A_{\theta}A_{\psi} \equiv AB$ for all $\theta, \psi \in S^{1}$ with $\psi$ lying at most $\frac{2\pi}{r}$ radians counterclockwise from $\theta$.
    \end{itemize}
\end{definition}

Now we translate Definition \ref{real-valued hierarchy, detailed definition}:

    \begin{definition}\label{real-valued hierarchy, hereditary class version}
    Let $r\geq 3$ be a real number. A hereditary class $\mathcal{H}$ of $\mathcal{L}$-structures is $\mathrm{NSOP}_{r}$ if the following equivalent conditions hold: 

    (1) For every finite $\mathcal{L}$-structure of the form $AB$ with a sequence $\{A_{i}\}_{i < \omega} \in \mathcal{H}$ such that $A_{i}A_{j} \equiv AB$ for all $i < j$, for all integers $N < \omega$ there are $A_{0}, \ldots, A_{N -1} $ such that $A_{i} A_{i+j \mathrm{\: mod \:} N} \equiv AB$ for $0 \leq i <N$, $1 \leq j \leq \frac{N}{r}$.

    (1$'$) For every finite $\mathcal{L}$-structure of the form $AB$ with a sequence $\{A_{i}\}_{i < \omega} \in \mathcal{H}$ such that $A_{i}A_{j} \equiv AB$ for all $i < j$, for all \textit{sufficiently large} integers $N < \omega$ there are $A_{0}, \ldots, A_{N -1} $ such that $A_{i} A_{i+j \mathrm{\: mod \:} N} \equiv AB$ for $0 \leq i <N$, $1 \leq j \leq \frac{N}{r}$.

    (2) For every finite $\mathcal{L}$-structure of the form $AB$ with a sequence $\{A_{i}\}_{i < \omega} \in \mathcal{H}$ such that $A_{i}A_{j} \equiv AB$ for all $i < j$, there exists a set $\{ A_{\theta}\}_{\theta \in S^{1}} \in \mathcal{H}$ indexed by the unit circle $S^{1}$ with $A_{\theta}A_{\psi} \equiv AB$ for all $\theta, \psi \in S^{1}$ with $\psi$ lying at most $\frac{2\pi}{r}$ radians counterclockwise from $\theta$.

 Otherwise it has $\mathrm{SOP}_{r}$.
\end{definition}

Remark \ref{rationals and irrationals} applies equally as well to this definition, and the properties $\mathrm{NSOP}_{r}$ for $r \geq 3$ form an ascending chain as in Remark 2.8. We see from Definitions \ref{real value of a hereditary class} and \ref{real-valued hierarchy, hereditary class version} that $\mathfrak{o}(\mathcal{H})$ is just the supremum of the values $r \geq 3$ such that $\mathcal{H}$ is $\mathrm{SOP}_{r}$.

\begin{remark}\label{hereditary classes and quantifier-free formulas}
    As mentioned in the beginning of this section, the $\mathrm{NSOP}_{r}$ hierarchy for hereditary classes is equivalent to the $\mathrm{NSOP}_{r}$ hierarchy \textit{at the quantifier-free} level for theories: let $T$ be a theory (say, in a finite language $\mathcal{L}$). Then the hereditary class $\mathcal{H}$ of finite $\mathcal{L}$-substructures of models of $T$ is $\mathrm{NSOP}_{r}$ if and only if $T$ is $\mathrm{NSOP}_{r}$ \textit{at the quantifier-free level}: for every \textit{quantifier-free} definable relation $R(x_{1}, x_{2})$ with a sequence $\{a_{i}\}_{i < \omega}$ such that $\models R(a_{i}, a_{j})$ for all $i < j$, there exists a set $\{a_{\theta}\}_{\theta \in S^{1}}$ such that $\models R(a_{\theta}, a_{\psi})$ for all $\theta, \psi \in S^{1}$ with $\psi$ lying at most $\frac{2\pi}{r}$ radians counterclockwise from $\theta$. To see this, note that the condition on $\mathcal{H}$ is equivalent to saying: for every definable relation $R(x_{1}, x_{2})$ \textit{describing a complete quantifier-free type} with a sequence $\{a_{i}\}_{i < \omega}$ such that $\models R(a_{i}, a_{j})$ for all $i < j$, there exists a set $\{a_{\theta}\}_{\theta \in S^{1}}$, indexed by the unit circle $S^{1}$, such that $\models R(a_{\theta}, a_{\psi})$ for all $\theta, \psi \in S^{1}$ with $\psi$ lying at most $\frac{2\pi}{r}$ radians counterclockwise from $\theta$. So the hereditary class condition implies the quantifier-free condition. The opposite direction is just the standard translation from formulas to types (see e.g. \cite{She95}): let the quantifier-free definable relation $R(x_{1}, x_{2})$, with a sequence $\{a_{i}\}_{i < \omega}$ such that $\models R(a_{i}, a_{j})$ for all $i < j$, exhibit $\mathrm{SOP}_{r}$. Extracting an indiscernible sequence, all of the $a_{i}a_{j}$ for $i < j$ can be assumed to have the same quantifier-free type, described by a formula $R'(x, y)$. But since $R'(x, y)$ is stronger than $R(x, y)$, there can be no $\{a_{\theta}\}_{\theta \in S^{1}}$ as in the definition of $\mathrm{SOP}_{r}$, so $R'(x, y)$ is a formula describing a complete quantifier-free type exhibiting $\mathrm{SOP}_{r}$, and the hereditary class condition fails.
\end{remark}

It is even open whether the real-valued and integer-valued hierarchies are distinct at this quantifier-free level:

\begin{question}\label{hereditary class distinctness}
    Let $\mathcal{H}$ be a hereditary class of $\mathcal{L}$-structures. Is $\mathfrak{o}(\mathcal{H})$ an integer? Is the $\mathrm{NSOP}_{r}$ hierarchy for $r \geq 3$ a real number distinct from the $\mathrm{NSOP}_{n}$ hierarchy for $n \geq 3$ an integer \emph{for hereditary classes}?
\end{question}

Note that the question of whether the hierarchies are distinct is slightly different from the question of whether $\mathfrak{o}(\mathcal{H})$ is an integer: it is conceivable that there is a hereditary class $\mathcal{H}$ such that, for some integer $n > 3$, $\mathcal{H}$ is $\mathrm{NSOP}_{n}$ but has $\mathrm{SOP}_{r}$ for all $r < n$. Then $\mathfrak{o}(\mathcal{H}) = n$, but the real-valued and integer-valued hierarchies are distinct, because all currently known $\mathrm{NSOP}_{n+1}$ hereditary classes with $\mathrm{SOP}_{n}$ (including all those defined by a finite family of forbidden weakly embedded substructures; see next section) are also $\mathrm{NSOP}_{r}$ for all real numbers $r > n$.

More generally, $\mathfrak{o}(\mathcal{H})$ determines whether $\mathcal{H}$ is $\mathrm{NSOP}_{r}$ for all real values $r \geq 3$, except for at most one: if $\mathfrak{o}(\mathcal{H}) = r $ is rational, it is conceivable that $\mathcal{H}$ either is $\mathrm{NSOP}_{r}$ or has $\mathrm{SOP}_{r}$. However, if $\mathfrak{o}(\mathcal{H}) = r$ is irrational, $\mathcal{H}$ must be $\mathrm{NSOP}_{r}$, by Remark \ref{rationals and irrationals}.

In later sections, we will consider Question \ref{hereditary class distinctness}. The next section will answer this question for all hereditary classes defied by a \textit{finite} family of forbidden weakly embeddded substructures. To give additional motivation for this result, we will show that Question \ref{hereditary class distinctness} can be reduced to the case of a hereditary class defined by a (potentially infinite) family of forbidden weakly embedded substructures.

Recall the definition of \textit{weak embedding}:

\begin{definition}\label{weak embedding}
    Let $\mathcal{L}$ be a finite relational language. Then a one-to-one map $\iota: A \hookrightarrow B$ between $\mathcal{L}$-structures is a \emph{weak embedding} if, for $R$ any relation symbol of $\mathcal{L}$,

    $$(a_{1}, \ldots, a_{n}) \in R(A) \Rightarrow (\iota(a_{1}), \ldots, \iota(a_{n})) \in R(B)$$.

\end{definition}

Note that the reverse implication does not hold.

Hereditary classes are often defined by a family of forbidden substructures (such as the class of $K_{n}$-free graphs); in what follows, we will be interested in hereditary classes defined by families of forbidden \textit{weakly embedded} substructures. For more on classes of structures defined by forbidden substructures, and their connections to model theory, see Cherlin, Shelah and Shi (\cite{CSS99}).

\begin{definition}\label{class of structures omitting a weakly embedded substructure}
    Let $\mathcal{F}$ be a family of finite structures in a finite relational language $\mathcal{L}$. We define $\mathcal{H}(\mathcal{F})$ to be the hereditary class of $\mathcal{L}$-structures $A$ such that there is no $B \in \mathcal{F}$ with a weak embedding $\iota: B \hookrightarrow A$.
\end{definition}

\begin{definition}\label{defined by a (finite) family of forbidden weakly embedded substructures}
    A hereditary class $\mathcal{H}$ of $\mathcal{L}$-structures is \textit{defined by a family of forbidden weakly embedded substructures} (or alternatively, is \textit{closed under weak embeddings}) if $\mathcal{H} = \mathcal{H}(\mathcal{F})$ for some family $\mathcal{F}$ of $\mathcal{L}$-structures (equivalently, is such that, if $A \in \mathcal{H}$ and $\iota: B \hookrightarrow A$ is a weak embedding, then $B \in \mathcal{H}$).

    A hereditary class $\mathcal{H}$ of $\mathcal{L}$-structures is \textit{defined by a finite family of forbidden weakly embedded substructures} if $\mathcal{H} = \mathcal{H}(\mathcal{F})$ for some finite family $\mathcal{F}$ of $\mathcal{L}$-structures.
\end{definition}

To see the equivalence of the two conditions for the first definition, if $\mathcal{H} = \mathcal{H}(\mathcal{F})$ for some family $\mathcal{F}$ of $\mathcal{L}$-structures, then $\mathcal{H}$ is closed under weak embeddings because weak embeddings are closed under composition. Conversely, if $\mathcal{H}$ is closed under weak embeddings, then $\mathcal{H} = \mathcal{H}(\mathcal{F})$, where $\mathcal{F}$ is the complement of $\mathcal{H}$ in the class of all $\mathcal{L}$-structures.

We now see that the question of whether $\mathfrak{o}(\mathcal{H})$ is an integer, or whether the real-valued and integer-valued hierarchies are distinct at the level of hereditary classes, reduces to the case of a hereditary class closed under weak embeddings:

\begin{prop}\label{reduction to the case of a hereditary class closed under weak embeddings}
    Let $\mathcal{H}$ be a hereditary class of $\mathcal{L}$-structures. Then there is some finite language $\mathcal{L}'$, and some hereditary class $\mathcal{H}'$ of $\mathcal{L}'$-structures defined by a family of forbidden weakly embedded substructures, such that $\mathfrak{o}(\mathcal{H})= \mathfrak{o}(\mathcal{H}')$, and in fact such that, for any real number $r \geq 3$, $\mathcal{H}$ is $\mathrm{NSOP}_{r}$ if and only if $\mathcal{H}'$ is $\mathrm{NSOP}_{r}$.
\end{prop}

\begin{proof}
   We first prove the following claim, which we use later in the proof:

   \begin{claim}\label{circle with disjoint terms}
        Let $\mathcal{H}$ be any $\mathrm{NSOP}_{r}$ hereditary class for $r \geq 3$ a real number. For every finite $\mathcal{L}$-structure of the form $AB$ with a sequence $\{A_{i}\}_{i < \omega} \in \mathcal{H}$ such that $A_{i}A_{j} \equiv AB$ for all $i < j$, and for $C = A \cap B$, there exists a set $\{ A_{\theta}\}_{\theta \in S^{1}} \in \mathcal{H}$ of $A_{\theta} \supset C$ (i.e., with $C$ as an induced substructure of $A_{i}$), with $A_{\theta}A_{\psi} \equiv_{C} AB$ for all $\theta, \psi \in S^{1}$ with $\psi$ lying at most $\frac{2\pi}{r}$ radians counterclockwise from $\theta$, and such that $A_{\theta} \cap A_{\psi} = C$ for all $\theta, \psi \in S^{1}$ with $\theta \neq \psi$.
   \end{claim}
\end{proof}

\begin{proof}
    We may assume without loss of generality that $C = \emptyset$; we apply compactness arguments throughout. Fixing an arbitrary finite subset $S \subset S^{1}$, it suffices to show that there exists a set $\{ A_{\theta}\}_{\theta \in S} \in \mathcal{H}$ with $A_{\theta}A_{\psi} \equiv AB$ for all $\theta, \psi \in S$ with $\psi$ lying at most $\frac{2\pi}{r}$ radians counterclockwise from $\theta$, and such that $A_{\theta} \cap A_{\psi} = \emptyset$ for all $\theta, \psi \in S^{1}$ with $\theta \neq \psi$. We may assume that $\{A_{i}\}_{i < \omega}$ is indiscernible (in the quantifier-free sense); then we can extend it to $\{A_{i}\}_{i < \omega \times \omega} \in \mathcal{H}$ such that $A_{i}A_{j} \equiv AB$ for all $i < j < \omega \times \omega$. Define $\{\overline{A}_{i}\}_{i < \omega} : = \{A_{i \omega} \ldots A_{i\omega + n}, \ldots\}_{i < \omega}$. By $\mathrm{NSOP}_{r}$, we obtain a set $\{ \overline{A}_{\theta}\}_{\theta \in S} \in \mathcal{H}$ with $\overline{A}_{\theta}\overline{A}_{\psi} \equiv \overline{A}_{0} \overline{A}_{1}$ for all $\theta, \psi \in S$ with $\psi$ lying at most $\frac{2\pi}{r}$ radians counterclockwise from $\theta$. Inductively, finding $A_{\theta} : = (A_{i})_{\theta}$ within the $\overline{A}_{\theta} =: (A_{0})_{\theta} \ldots (A_{n})_{\theta} \ldots $ for properly chosen $i$ by a pigeonhole principle argument, we will be able to obtain $\{ A_{\theta}\}_{\theta \in S}$ as desired.
\end{proof}

Now, in the language $\mathcal{L}^{*}$ consisting of the symbols of $\mathcal{L}$ along with, for each $n$-ary $R \in \mathcal{L}$, a new $n$-ary symbol symbol $\neg_{R}$, define a new hereditary class $\mathcal{H}^{*}$ as follows: $\mathcal{H}^{*}$ consists of, for each $A \in \mathcal{H}$, the structure $A^{*}$ expanding $A$ by the predicates $\neg^{R}$ for each $n$-ary relation $R \in \mathcal{L}$, where $(a_{1}, \ldots a_{n}) \in \neg^{R}(A^{*})$ exactly when $(a_{1}, \ldots a_{n}) \notin R(A^{*})= R(A)$. Then it follows from the assumption that $\mathcal{H}$ is $\mathrm{NSOP}_{r}$ that $\mathcal{H}^{*}$ is $\mathrm{NSOP}_{r}$. Moreover, $\mathcal{H}^{*}$ has the property that every surjective weak embedding between structures in $\mathcal{H}^{*}$ is an isomorphism.

Therefore, we can assume that the hereditary class $\mathcal{H}$ of $\mathcal{L}$-structures has this property:

(*) Let $A, B \in \mathcal{H}$, and let $\iota: A \hookrightarrow B$ be a surjective weak embedding. Then $\iota$ is an isomorphism.

Let $\mathcal{H}'$ consist of those $\mathcal{L}$-structures $A$ such that there exists a weak embedding $\iota: A \hookrightarrow B$ into some $B \in \mathcal{H}$. Then $\mathcal{H}'$ is closed under weak embeddings, and we show $\mathcal{H}'$ is as desired.

First, suppose $\mathcal{H}$ has $\mathrm{SOP}_{r}$. We show that $\mathcal{H}'$ has $\mathrm{SOP}_{r}$. Let $AB$ be such that there exists a sequence $\{ A_{i}\}_{i < \omega} \in \mathcal{H}$ such that $A_{i}A_{j} \equiv AB$ for all $i < j$, but there does not exist a set $\{ A_{\theta}\}_{\theta \in S^{1}} \in \mathcal{H}$ with $A_{\theta}A_{\psi} \equiv AB$ for all $\theta, \psi \in S^{1}$ with $\psi$ lying at most $\frac{2\pi}{r}$ radians counterclockwise from $\theta$. Then $\{ A_{i}\}_{i < \omega} \in \mathcal{H}'$, so it remains to show there does not exist a set $\{ A_{\theta}\}_{\theta \in S^{1}} \in \mathcal{H}'$ with $A_{\theta}A_{\psi} \equiv AB$ for all $\theta, \psi \in S^{1}$ with $\psi$ lying at most $\frac{2\pi}{r}$ radians counterclockwise from $\theta$. But suppose there was; then there is some weak embedding $\iota: \{ A_{\theta}\}_{\theta \in S^{1}} \hookrightarrow \overline{A} \in \mathcal{H}$, by the definition of $\mathcal{H}'$ and compactness. Then, $ \{ \iota(A_{\theta})\}_{\theta \in S^{1}}=\iota(\{ A_{\theta}\}_{\theta \in S^{1}}) \in \mathcal{H}$ , and by (*), $\iota(A_{\theta})\iota(A_{\psi}) \equiv AB$ for all $\theta, \psi \in S^{1}$ with $\psi$ lying at most $\frac{2\pi}{r}$ radians counterclockwise from $\theta$--contradicting our assumption on $AB$.

Conversely, suppose $\mathcal{H}$ is $\mathrm{NSOP}_{r}$, and let $AB$ be such that there exists a sequence $\{ A_{i}\}_{i < \omega} \in \mathcal{H}'$ such that $A_{i}A_{j} \equiv AB$ for all $i < j$. We show that there exists a set $\{ A_{\theta}\}_{\theta \in S^{1}} \in \mathcal{H}$ with $A_{\theta}A_{\psi} \equiv AB$ for all $\theta, \psi \in S^{1}$ with $\psi$ lying at most $\frac{2\pi}{r}$ radians counterclockwise from $\theta$. Let $\iota: \{ A_{i}\}_{i < \omega} \hookrightarrow \overline{A}$ be a weak embedding where $\overline{A} \in \mathcal{H}$. Then $ \{ \iota(A_{\theta})\}_{\theta \in S^{1}}=\iota(\{ A_{\theta}\}_{\theta \in S^{1}}) \in \mathcal{H}$, and by replacing $\{ A_{i}\}_{i < \omega}$ with a subsequence (in which case $A_{i}A_{j} \equiv AB$ for all $i < j$ still), by Ramsey's theorem we can assume that $\iota(A_{i})\iota(A_{j}) \equiv \iota(A_{i'})\iota(A_{j'})$ whenever $i < j$ and $i' < j'$. By $\mathrm{NSOP}_{r}$ and the claim, for $C' = \iota(A_{0}) \cap \iota(A_{1})$ we can find $\{ A'_{\theta}\}_{\theta \in S^{1}} \in \mathcal{H}$ with $A'_{\theta} \supset C'$ and $A'_{\theta}A'_{\psi} \equiv_{C} \iota(A_{0})\iota(A_{1})$ for all $\theta, \psi \in S^{1}$ with $\psi$ lying at most $\frac{2\pi}{r}$ radians counterclockwise from $\theta$, and such that $A'_{\theta} \cap A'_{\psi} = C'$ for all $\theta, \psi \in S^{1}$ with $\theta \neq \psi$. By this last disjointness condition, we can define a new $\mathcal{L}$-structure, $\{ A_{\theta}\}_{\theta \in S^{1}}$, on the same underlying set as $\{ A'_{\theta}\}_{\theta \in S^{1}}$, to be the unique $\mathcal{L}$-structure satisfying the following conditions:

(1) For all $\theta, \psi \in S^{1}$ with $\psi$ lying at most $\frac{2\pi}{r}$ radians counterclockwise from $\theta$, $A_{\theta}A_{\psi} \equiv AB$.

(2) Any tuple with coordinates in (the union of) $\{ A_{\theta}\}_{\theta \in S^{1}}$, whose coordinates are not all contained in some fixed pair of the form $A_{\theta}A_{\psi}$ with $\psi$ lying at most $\frac{2\pi}{r}$ radians counterclockwise from $\theta$, belongs to no relation of $\mathcal{L}$.

Then, because $\iota|_{A_{0}A_{1}}$ is a weak embedding, the identity map $\mathrm{id}: \{ A_{\theta}\}_{\theta \in S^{1}} \hookrightarrow \{ A'_{\theta}\}_{\theta \in S^{1}}$ is a weak embedding (i.e.,  $\{ A_{\theta}\}_{\theta \in S^{1}}$ is a \textit{weak substructure} of $\{ A'_{\theta}\}_{\theta \in S^{1}}$). But then, since $\{ A_{\theta}\}_{\theta \in S^{1}} \in \mathcal{H}$, $\{ A'_{\theta}\}_{\theta \in S^{1}} \in \mathcal{H}'$, so by (1) is as desired.

\section{Integrality of $\mathfrak{o}(\mathcal{H}(\mathcal{F}))$ for $\mathcal{F}$ finite}

From Proposition \ref{reduction to the case of a hereditary class closed under weak embeddings}, we know that, if there are non-integer values of $\mathfrak{o}(\mathcal{H})$, then there are non-integer values of $\mathfrak{o}(\mathcal{H})$ for $\mathcal{H}$ a hereditary class defined by a family of forbidden weakly embedded substructures. By contrast, when $\mathcal{H}$ is defined by a \textit{finite} family of forbidden weakly embedded substructures, $\mathfrak{o}(\mathcal{H})$ must be an integer:

\begin{theorem} \label{integerality}

Let $\mathcal{H}$ be a hereditary class defined by a finite family of forbidden weakly embedded substructures. Then $\mathfrak{o}(\mathcal{H})$ is an integer.
\end{theorem}

In fact, we show that, if $\mathcal{H}$ has $\mathrm{SOP}_{r}$ for $r \geq 3$ a real number, then $\mathcal{H}$ must have $\mathrm{SOP}_{n}$ for $n = \lceil r \rceil $ the next integer. Therefore, there is no hereditary class, defined by a finite family of forbidden weakly embedded substructures, exhibiting that the $\mathrm{NSOP}_{r}$ hierarchy for real $r \geq 3$ is distinct from the $\mathrm{NSOP}_{n}$ hierarchy for $n \geq 3$ at the level of hereditary classes (question \ref{hereditary class distinctness}).

\begin{example}\label{examples of hereditary classes}
    For $n \geq 3$ an integer, there is a hereditary class $\mathcal{H}$, defined by a finite family of forbidden weakly embedded substructures, such that $\mathcal{H}$ is $\mathrm{NSOP}_{n+1}$ but has $\mathrm{SOP}_{n}$: for example, let $\mathcal{H}$ be the hereditary class consisting of directed graphs with no directed $n$-cycles for $k \leq n$. This is essentially Claim 2.8 of \cite{She95}, and the same proof works for the hereditary class consisting of directed graphs with no directed $n$-cycles. By Theorem \ref{integerality} (or really, the remark after it), $\mathfrak{o}(\mathcal{H}) = n$.

    However, by the same theorem, these examples can be misleading. Let $\mathcal{H}$ be the hereditary class of directed graphs with no directed $(N, k)$-cycles (as in Notation \ref{(N, k)-cycle}). Then (when $\frac{N}{k} \geq 3$), one might naively believe that $\mathfrak{o}(\mathcal{H})$ is equal to $\frac{N}{k}$. Of course, by Theorem \ref{integerality}, $\mathfrak{o}(\mathcal{H})$ must be an integer, leaving open Question \ref{hereditary class distinctness} on whether $\mathfrak{o}(\mathcal{H})$ is an integer in general.
\end{example}

To prove Theorem \ref{integerality}, we first reduce to the case of a hereditary class $\mathcal{H}$ of directed graphs, where the failure of $\mathfrak{o}(\mathcal{H})$ to be an integer is exhibited by the graph relation itself.

\begin{lemma}\label{reduction to directed graph defined by a finite family of forbidden weakly embedded substructures}
    Suppose that there is a hereditary class $\mathcal{H}$, defined by a finite family of forbidden weakly embedded substructures, that has $\mathrm{SOP}_{r}$ for $r \geq 3$ a real number, but is $\mathrm{NSOP}_{n}$ for $n = \lceil r \rceil $ the next integer. Then there is an $\mathrm{NSOP}_{n}$ hereditary class $\mathcal{H}'$ of directed graphs with edge relation $R$, defined by a finite family of forbidden weakly embedded substructures, such that the graph $\{a_{i}\}_{i < \omega} \in \mathcal{H}'$ where $a_{i} R a_{j}$ exactly when $i < j$, but $\{a_{\theta}\}_{\theta \in S^{1}} \notin \mathcal{H}'$ where $a_{\theta} R a_{\psi}$ exactly when $\psi$ lies at most $\frac{2\pi}{r}$ radians counterclockwise from $\theta$.
\end{lemma}

\begin{proof}
    Let the structure $AB$ exhibit $\mathrm{SOP}_{r}$ for the hereditary class $\mathcal{H}$ of $\mathcal{L}$-structures, so there exists a sequence $\{ A_{i}\}_{i < \omega} \in \mathcal{H}$ such that $A_{i}A_{j} \equiv AB$ for all $i < j$, but there does not exist a set $\{ A_{\theta}\}_{\theta \in S^{1}} \in \mathcal{H}$ with $A_{\theta}A_{\psi} \equiv AB$ for all $\theta, \psi \in S^{1}$ with $\psi$ lying at most $\frac{2\pi}{r}$ radians counterclockwise from $\theta$. Let $C := A_{0} \cap A_{1}$, and $A^{*}_{i} = A_{i} \backslash C$; ; let $k = |A_{i} \backslash C|$. Define the hereditary class $\mathcal{H}'$ of directed graphs with edge relation $R$ as follows: a directed graph $G = \{v_{1}, \ldots, v_{\ell}\} \in \mathcal{H}'$ if $\tilde{G} \in \mathcal{H}$, where $\tilde{G}$ is the unique $\mathcal{L}$-structure with underlying set $C \sqcup \bigsqcup^{\ell}_{i=1} \tilde{A}_{i}$ containing $C$ as an induced substructure, such that the $\tilde{A}_{i}$ are $k$-tuples with $\tilde{A}_{i} \equiv_{C} A^{*}_{0}$, $\tilde{A}_{i}\tilde{A}_{j} \equiv_{C} A^{*}_{0} A^{*}_{1}$ whenever $v_{i} R v_{j}$, and any tuple with coordinates in $\tilde{G} = C \sqcup \bigsqcup^{\ell}_{i=1} \tilde{A}_{i}$, whose coordinates are not all contained in some fixed set of the form $\tilde{A}_{i} C$, or of the form $\tilde{A}_{i} \tilde{A}_{j} C$ for $v_{i} R v_{j}$, belongs to no relation of $\mathcal{L}$.

    We first show that $\mathcal{H}'$ really is defined by a finite family of forbidden weakly embedded substructures. Let $\mathcal{H} = \mathcal{H}(\mathcal{F})$ for $\mathcal{F}$ finite, and let $N$ be the size of the largest structure in $\mathcal{F}$. We show that $\mathcal{H}' = \mathcal{H}(\mathcal{F}')$ where $\mathcal{F}'$ is the collection of all directed graphs of size at most $n$ not belonging to $\mathcal{H}$. First, $\mathcal{H}' \subseteq \mathcal{H}(\mathcal{F}')$, because $\mathcal{H}'$ is closed under weak embeddings, because, by construction, any weak embedding of graphs $\iota: G \hookrightarrow H$ gives a weak embedding of $\mathcal{L}$-structures $\tilde{\iota}: \tilde{G} \hookrightarrow \tilde{H}$, and $\mathcal{H}$ is closed under weak embeddings. We now show $\mathcal{H}' \supseteq \mathcal{H}(\mathcal{F}')$; suppose $G= \{v_{1}, \ldots, v_{\ell}\} \notin \mathcal{H}'$. Then $\tilde{G} \notin \mathcal{H}$. So $\tilde{G}= C \sqcup \bigsqcup^{\ell}_{i=1} \tilde{A}_{i}$ contains some $D \in \mathcal{F}$, which will then be contained in $C \sqcup \bigsqcup_{i \in S} \tilde{A}_{i}$ for $S \subset \{1, \ldots \ell\}$ of size at most $N$. Then for $G_{0} := \{v_{i}\}_{i \in S} \subset G$, $\tilde{G_{0}} = C \sqcup \bigsqcup_{i \in S} \tilde{A}_{i} \notin \mathcal{H}$, so $G_{0} \in \mathcal{F}'$ and $G \notin \mathcal{H}(\mathcal{F}')$.

    Next, we show that $G=\{a_{i}\}_{i < \omega} \in \mathcal{H}'$ where $a_{i} R a_{j}$ exactly when $i < j$, but $H=\{a_{\theta}\}_{\theta \in S^{1}} \notin \mathcal{H}'$ where $a_{\theta} R a_{\psi}$ exactly when $\psi$ lies at most $\frac{2\pi}{r}$ radians counterclockwise from $\theta$. But $G \in \mathcal{H}'$: $\tilde{G}$ weakly embeds in $\{ A_{i}\}_{i < \omega} \in \mathcal{H}$, so $\tilde{G} \in \mathcal{H}$ because $\mathcal{H}$ is closed under weak embeddings. And $H \notin \mathcal{H}'$, because otherwise $\tilde{H} \in \mathcal{H}$ would be a set $\{ A_{\theta}\}_{\theta \in S^{1}} \in \mathcal{H}$ with $A_{\theta}A_{\psi} \equiv AB$ for all $\theta, \psi \in S^{1}$ with $\psi$ lying at most $\frac{2\pi}{r}$ radians counterclockwise from $\theta$.

    Finally, we show $\mathcal{H}'$ is $\mathrm{NSOP}_{n}$. Let $\{G_{i}\}_{i < \omega} \in \mathcal{H}'$ be such that $G_{0} G_{1} \equiv G_{i}G_{j} $ for $i < j$. Let $G = \tilde{G}_{0} \cap \tilde{G}_{1}$. So $\widetilde{\{G_{i}\}_{i < \omega}}=\{\tilde{G}_{i}\}_{i < \omega} \in \mathcal{H}$ with $\tilde{G}_{0} \tilde{G}_{1} \equiv \tilde{G}_{i}\tilde{G}_{j} $ for $i < j$., and $\tilde{G}_{0} \cap \tilde{G}_{1} = \tilde{G}$. Since $\mathcal{H}$ is $\mathrm{NSOP}_{n}$, by the proof of Claim \ref{circle with disjoint terms}, we have $\{G^{*}_{i}\}^{n-1}_{i =0} \in \mathcal{H}$ with $G^{*}_{i}G^{*}_{i+1 \mathrm{\: mod \:}n}\equiv \tilde{G}_{0} \tilde{G}_{1} $ for $i < n$, and $G^{*}_{i} \cap G^{*}_{j} = \tilde{G}$ for $i \neq j$. Define the $H$ to be the unique directed graph of the form  $G \sqcup \bigsqcup^{n-1}_{i=0} H_{i}$ with $G$ as an induced substructure, satisfying the following conditions:

    (1) $(H_{i}G)(H_{i+1 \mathrm{\: mod \:}n}G)\equiv G_{0} G_{1} $ for $i < n$.

    (2) There are no edges in $H$ that are not between vertices of a pair of the form $H_{i}H_{i+1 \mathrm{\: mod \:}n}G$ for $i < n$.

    By (1), to show $\mathcal{H}'$ is $\mathrm{NSOP}_{n}$, because  $\{G_{i}\}_{i < \omega}$ was an arbitrary sequence in $\mathcal{H}'$ with $G_{0} G_{1} \equiv G_{i}G_{j} $ for $i < j$, it remains to show $H = \{H_{i}G\}^{n-1}_{i=0} \in \mathcal{H}'$. But $\tilde{H}$ weakly embeds in $\{G^{*}_{i}\}^{n-1}_{i =0} \in \mathcal{H}$, using the disjointness condition on the $G^{*}_{i}$, so $\tilde{H} \in \mathcal{H}$ because $\mathcal{H}$ is closed under weak embedding, and $H \in \mathcal{H}'$.
\end{proof}

\begin{remark}\label{reduction to directed graphs}
    Using the above arguments, we can extend Proposition \ref{reduction to the case of a hereditary class closed under weak embeddings}, showing that $\mathcal{H}'$ in that proposition can even be chosen to a hereditary class of \textit{directed graphs} closed under weak embeddings. Moreover, if $\mathcal{H}$ is $\mathrm{SOP}_{s}$ for some real number $s \geq 3$, then $\mathcal{H}'$ can be chosen so that its graph relation $R$ exhibits $\mathrm{SOP}_{s}$: $\{a_{i}\}_{i < \omega} \in \mathcal{H}'$ where $a_{i} R a_{j}$ exactly when $i < j$, but $\{a_{\theta}\}_{\theta \in S^{1}} \notin \mathcal{H}'$ where $a_{\theta} R a_{\psi}$ exactly when $\psi$ lies at most $\frac{2\pi}{s}$ radians counterclockwise from $\theta$. . So distinctness of the real-valued and the integer-valued hierarchies at the level of hereditary classes, or integerality of $\mathfrak{o}(\mathcal{H})$, really reduces to the case of a hereditary class of directed graphs closed under weak embeddings--and in fact, if the hierarchies are distinct, there is such a hereditary class whose graph relation exhibits $\mathrm{SOP}_{r}$ for some real number $r \geq 3$, but which is $\mathrm{NSOP}_{n}$ for $n = \lceil r \rceil$ the next integer.
\end{remark}

Before proceeding, we briefly rule out the edge case that $\mathfrak{o}(\mathcal{H})$ is infinite.

\begin{fact}
    If $\mathcal{H}$ is a hereditary class of graphs defined by a finite family of forbidden weakly embedded subgraphs, then $\mathfrak{o}(\mathcal{H})$ is finite.
\end{fact}

\begin{proof}
    We show that $\mathcal{H}=\mathcal{H}(\mathcal{F})$ is $\mathrm{NSOP}_{n}$ for $n$ larger than the size of the largest structure in $\mathcal{F}$. Let $\{G_{i}\}_{i < \omega} \in \mathcal{H}$ be such that $G_{0}G_{1} \equiv G_{i}G_{j}$ for $i < j$, and let $H= G \sqcup \bigsqcup^{n-1}_{i=0} H_{i}$  for $G = G_{0} \cap G_{1}$ satisfy conditions (1) and (2) as in the proof of Lemma \ref{reduction to the case of a hereditary class closed under weak embeddings}. It remains only to show $H \in \mathcal{H}$. But otherwise, $H$ must contain some weakly embedded image of some structure in $\mathcal{F}$, which by choice of $n$ must belong to $ G \sqcup \bigsqcup_{i \in S} H_{i}$ for $S \subsetneq \{0, \ldots n-1\}$. But $ G \sqcup \bigsqcup_{i \in S} H_{i}$, and thus this element of $\mathcal{F}$, can be weakly embedded in $\{G_{i}\}_{i < \omega}$, so $\{G_{i}\}_{i < \omega} \notin \mathcal{H}$, a contradiction.
\end{proof}

Next, we make explicit an already-established construction in Shelah's $\mathrm{NSOP}_{n}$ hierarchy, the \textit{helix maps}. Shelah used this construction for reasoning about specific examples in his original introduction of the hierarchy in \cite{She95} (in Claim 2.8 of that article), but Malliaris first developed it at the abstract level in the proof of Theorem 7.7 of \cite{Mal10b}, where it proved crucial for understanding edge distribution in $\mathrm{NSOP}_{3}$ theories.

We will only need to discuss helix maps in the setting of hereditary classes of directed graphs closed under weak embeddings (the importance of which is suggested by Remark \ref{reduction to directed graphs}), but the following background can be straightforwardly generalized to any hereditary class of $\mathcal{L}$-structures closed under weak embeddings.

We start with the following definition, which gives us the underlying data for constructing helix maps:

\begin{definition}\label{cyclic n-decomposition}
    Let $G$ be a directed graph, and $n \geq 3$ be an integer. A \textit{cyclic $n$-decomposition} of $G$ is a partition $G = \bigsqcup^{n-1}_{i =0} G_{i} \sqcup D$ where every edge of $G$ has both endpoints in some fixed set of the form $G_{i}G_{i+1 \mathrm{\: mod \: n}}D$. A cyclic $n$-decomposition of $G$ is \textit{full} if $D = \emptyset$, and it is \textit{directed} if any edge between $v \in G_{i}$ and $v' \in G_{i + 1 \mathrm{\: mod\:  } n}$ for $i < n$ is in the direction from $v$ to $v'$.
\end{definition}

We will consider full and directed cyclic $n$-decompositions in the following section.

For clarity, in defining the helix map associated to any cyclic $n$-decomposition of a graph, we will make the enumeration of each set in the partition explicit, though later on we may refer to these sets using standard model-theoretic notation as described in Subsection \ref{hereditary class subsection}. 

\begin{definition}\label{helix map}
    Let $G = \bigsqcup^{n-1}_{i =0} G_{i} \sqcup D$ be a cyclic $n$-decomposition of a graph $G$. Enumerate each $G_{i}$ as $\{g_{i}(0), \ldots, g_{i}(n_{i})\}$ and $D$ as $\{d(0), \ldots, d(n_{d}))\}$. Let $H$ be the graph with distinct vertices $g^{j}_{i}(k)$, for $j < \omega$, $i < n$ and $k \leq n_{i}$, along with the vertices $d(0) \ldots, d(n_{d})$ of $D$ forming a subgraph of $H$, with the unique edge relation satisfying the following conditions:

    (1) For $j < j'$, and $i' = i + 1 \mathrm{\: mod \:} n$, the map from $\{g^{j}_{i}(0)\ldots g^{j}_{i}(n_{i})\} \sqcup \{g^{j'}_{i'}(0)\ldots g^{j'}_{i'}(n_{i'})\} \sqcup \{d(0), \ldots, d(n_{d}))\}$ to $G_{i} \sqcup G_{i'} \sqcup D$ given by $g^{j}_{i}(k) \mapsto g_{i}(k)$, $g^{j'}_{i'}(k) \mapsto g_{i'}(k)$, $d(k) \mapsto d(k)$ is an isomorphism.

    (2) There are no edges in $H$ whose endpoints are not both in some fixed set of the form $\{g^{j}_{i}(0)\ldots g^{j}_{i}(n_{i})\} \sqcup \{g^{j'}_{i'}(0)\ldots g^{j'}_{i'}(n_{i'})\} \sqcup \{d(0), \ldots, d(n_{d}))\}$, for $j < j'$, and $i' = i + 1 \mathrm{\: mod \:} n$.

    Then the \emph{$n$-helix map} associated with the cyclic $n$-decomposition $G = \bigsqcup^{n-1}_{i =0} G_{i} \sqcup D$ is the map $h: H \twoheadrightarrow G$ given by $g^{j}_{i}(k) \mapsto g_{i}(k)$, $d(k) \mapsto d(k)$. For $\ell > 1$, the \emph{$n$-helix map of length $\ell$} associated with this cyclic $n$-decomposition is the map $h^{\ell}: H^{\ell} \twoheadrightarrow G$ that is defined similarly, where $H^{\ell}$ is defined similarly to $H$ but for $j < \ell$ instead of $j < \omega$.
\end{definition}

The main idea of the proof of the following is found in the proof of Theorem 7.7 of \cite{Mal10b}. 

\begin{fact}\label{closure under helix maps} \emph{(Malliaris, \cite{Mal10b})} 

Let $\mathcal{H}$ be a hereditary class of directed graphs that is closed under weak embeddings. Then for $n \geq 3$ an integer, $\mathcal{H}$ is $\mathrm{NSOP}_{n}$ if it is closed under $n$-helix maps: if $h: H \twoheadrightarrow G$ is an $n$-helix map, and $H \in \mathcal{H}$, then $G \in \mathcal{H}$.
\end{fact}

Of course, the actual map $h$ in Definition \ref{helix map} is not important here, but it will be important later in our proof of Theorem \ref{integerality}.

\begin{proof}
    We first show that, if $\mathcal{H}$ is closed under $n$-helix maps, then $\mathcal{H}$ is $\mathrm{NSOP}_{n}$. Suppose $\mathcal{H}$ is closed under $n$-helix maps and let $AB \in \mathcal{H}$ have $\{A_{i}\}_{i < \omega}$ such that $A_{i} A_{j} \equiv AB$ for $i < j$, and define $D : = A_{i} \cap A_{j}$. Let us form $G = \bigsqcup^{n-1}_{i =0} G_{i} \sqcup D$ as in the end of the proof of Lemma \ref{reduction to directed graph defined by a finite family of forbidden weakly embedded substructures}: $(G_{i}D)(G_{i+1 \mathrm{\: mod \:}n}D)\equiv AB $ for $i < n$, while there are no edges whose endpoints do not both belong to some pair of the form $(G_{i}D)(G_{i+1 \mathrm{\: mod \:}n}D)$; then it suffices to show $G \in \mathcal{H}$. The partition $G = \bigsqcup^{n-1}_{i =0} G_{i} \sqcup D$ is a cyclic $n$-decomposition; let $h: H \twoheadrightarrow G$ be the $n$-helix map associated with this cyclic $n$-decomposition. For $i < n$, $j < \omega$, let $G^{j}_{i}$ denote $\{g^{j}_{i}(0)\ldots g^{j}_{i}(n_{i})\}$ as in Definition \ref{helix map} ($D$ is included in $H$ as $\{d(0), \ldots, d(n_{d}))\}$). Then $H$ weakly embeds in $\{A_{i}\}_{i < \omega} \in \mathcal{H}$ by the map that restricts, on $G^{j}_{i}D$, to the isomorphism (restricting to the identity on $D$) by which $G^{j}_{i}D \equiv A^{jn+ i}$. So $H \in \mathcal{H}$, then $G \in \mathcal{H}$ because $\mathcal{H}$ is closed under helix maps.

    It remains to show that, if $\mathcal{H}$ is $\mathrm{NSOP}_{n}$, then it is closed under $n$-helix maps. Let $G = \bigsqcup^{n-1}_{i =0} G_{i} \sqcup D$ be a cyclic $n$-decomposition of a graph $G$ and let $h: H \twoheadrightarrow G$ be the associated helix map. Suppose $\mathcal{H}$ is $\mathrm{NSOP}_{n}$, and $H \in \mathcal{H}$; we show $G \in \mathcal{H}$. Write $H$ as $\{G^{j}_{0}\ldots G^{j}_{n-1}D\}_{j<\omega}$. Then, by $\mathrm{NSOP}_{n}$ and the proof of Claim \ref{circle with disjoint terms}, we may find a graph $K =\bigsqcup_{i, j < n}K^{j}_{i} \sqcup D \in \mathcal{H}$ such that, for $j < n$, 
    
    $$K^{j}_{0}\ldots K^{j}_{n-1}K^{j+1 \mathrm{\:mod\:}n }_{0}\ldots K^{j+1 \mathrm{\:mod\:}n }_{n-1} \equiv_{D} G^{0}_{0}\ldots G^{0}_{n-1} G^{1}_{0}\ldots G^{1}_{n-1} $$.

    In particular, for $i< n$,

    $$K^{i}_{i}K^{i + 1 \mathrm{\:mod\:}n}_{i + 1 \mathrm{\:mod\:}n} \equiv_{D} G^{0}_{i}G^{1}_{i+1}\equiv_{D} G_{i}G_{i + 1 \mathrm{\:mod\:}n}$$

    where the second equivalence is given by the construction of the $n$-helix map. So because $G = \bigsqcup^{n-1}_{i =0} G_{i} \sqcup D$ is a cyclic $n$-decomposition, $G$ weakly embeds into $K$ by the map that is the identity on $D$ and, on $G_{i}$, restricts to the isomorphism exhibiting $G_{i} \equiv_{D} K^{i}_{i} $. (This is the ``diagonal argument" used in the proof of Theorem 7.7 of \cite{Mal10b}; see in particular figure 2 of that paper.) Therefore, because $K \in \mathcal{H}$, $G \in \mathcal{H}$.
\end{proof}

Now that we have defined helix maps, we give an abstract property of these maps in the category of graphs, relative to the cycles of minimal length in a graph. We recall the definition of a homomorphism of directed graphs:

\begin{definition}\label{graphhomomorphism}(See e.g. \cite{HN04})
 A map $f: G \to H$ between directed graphs is a \emph{graph homomorphism} if, for $v_{1}, v_{2} \in G$, $v_{1} R v_{2}$ implies $f(v_{1})Rf(v_{2})$ (in particular, implies $f(v_{1}) \neq f(v_{2})$).

\end{definition}

Then the following is immediate from the definition of a helix map:

\begin{prop}\label{helix maps are graph homomorphisms}

Any $n$-helix map (of length $\ell$) is a graph homomorphism.
    
\end{prop}

Depending on context, we may view directed $n$-cycles in a graph $G$ not just as lists of vertices in $G$, but as graph homomorphisms $ \gamma \to G$ where $\gamma = \{\gamma_{0}, \ldots, \gamma_{n-1}\}$ and $\gamma_{i} R \gamma_{j}$ if and only if $j = i+1 \mathrm{\: mod \:} n$ (and where, abusing notation, we may use $\gamma$ to refer to this particular graph, or to the map itself). In the rest of the section, we use ``cycle" and ``directed cycle" interchangeably.

\begin{prop}\label{cycle removal}
    
Let $\mathcal{H}$ be an $\mathrm{NSOP}_{n}$ hereditary class of directed graphs for $n \geq 3$ an integer, and let $g: H \to G$ be a graph homomorphism between finite graphs with $H, G \notin \mathcal{H}$. Let $\gamma \hookrightarrow G$ be a directed cycle in $G$ of minimal length, and suppose $k :=|\gamma| \geq n$. Then there is a $k$-helix map $h: \tilde{H} \twoheadrightarrow H$ of some length $\ell$ such that $\tilde{H} \notin \mathcal{H}$ and such that there is no $k$-cycle $\gamma' \hookrightarrow \tilde{H}$ with $\gamma = g\circ h\circ \gamma'$.

$$\begin{tikzcd}
                        &  & H \arrow[dd, "g"] &  & \tilde{H} \arrow[ll, two heads, "h", dotted] \arrow[lldd] \\
                        &  &              &  &                                      \\
\gamma \arrow[rr, hook] &  & G            &  &                                     
\end{tikzcd}$$

\end{prop}

\begin{proof}
    We first prove the following claim about cycles of minimal length.

    \begin{claim}\label{minimal length cycles}
        Let $\gamma \hookrightarrow G$ be a directed cycle in $G$ of minimal length. Then $\gamma$ is an isomorphism onto its image (so in particular, is one-to-one, justifying the notation $\gamma \hookrightarrow G$.)
    \end{claim}
\begin{proof} (of claim)

First of all, $\gamma: \{\gamma_{0}, \ldots, \gamma_{k-1}\} \to G$ must be one-to-one, because if $\gamma(\gamma_{i}) = \gamma(\gamma_{j})$ for $i < j$, then $\gamma(\gamma_{i}), \ldots, \gamma(\gamma_{j-1})$ must be a cycle of smaller length, a contradiction. Suppose $\gamma$ is not an isomorphism onto its image. Then for some $i < j$ where $j \neq i+1 \mathrm{\: mod \:} n$ and $i \neq j+1 \mathrm{\: mod \:} n$, there is an edge (in either direction) between $\gamma(\gamma_{i})$ and  $\gamma(\gamma_{j})$. Depending on the direction, either $\gamma(\gamma_{i}), \ldots \gamma(\gamma_{j})$ will then be a cycle of smaller length, or $\gamma(\gamma_{j}) \ldots \gamma(\gamma_{k-1}) \ldots \gamma(\gamma_{1}) \ldots \gamma(\gamma_{i})$ will be a cycle of smaller length, again a contradiction.
    
\end{proof}

By the claim and the assumption that $g$ is a graph homomorphism, $H = \bigsqcup^{k-1}_{i =0} G_{i} \sqcup D$ is a cyclic $k$-decomposition of $H$, where $G_{i} := g^{-1}(\{\gamma(\gamma_{i})\})$ and $D := H \backslash \bigsqcup^{k-1}_{i =0} G_{i}  $. Since $k \geq n$, $\mathcal{H}$ is $\mathrm{NSOP}_{k}$. By Fact \ref{closure under helix maps}, because $H \notin \mathcal{H}$, for some $\ell$, $\tilde{H} \notin \mathcal{H}$ where $h : \tilde{H} \twoheadrightarrow H$ is the $k$-helix map of length $\ell$ associated with the cyclic $k$-decomposition $H = \bigsqcup^{k-1}_{i =0} G_{i} \sqcup D$. To show that $h : \tilde{H} \twoheadrightarrow H$ is as desired, it remains to show that there is no $k$-cycle $\gamma' \hookrightarrow \tilde{H}$ with $\gamma = g\circ h\circ \gamma'$. But the condition $\gamma = g\circ h\circ \gamma'$ implies that $g(h(\gamma'(\gamma_{i})))= \gamma(\gamma_{i})$, so $h(\gamma'(\gamma_{i})) \in G_{i}$, so $\gamma'(\gamma_{i}) \in G^{j_{i}}_{i}$ (using the notation of the proof of Fact \ref{closure under helix maps}) for some $j_{i} < \ell$. For $i < k-1$, since $\gamma'(\gamma_{i}) R \gamma'(\gamma_{i+1})$, $j_{i+1} > j_{i}$, because the only edges between $G_{i}$ and $G_{i+1}$ go from $G_{i}$ to $G_{i+1}$ due to $g, \gamma$ being graph homomorphisms, and by construction of the $k$-helix map associated with $H = \bigsqcup^{k-1}_{i =0} G_{i} \sqcup D$. So $j_{0} < \ldots < j_{k-1}$, and $\gamma'(\gamma_{0}) \in G^{j_{1}}_{0}$ while $\gamma'(\gamma_{k-1}) \in G^{j_{k-1}}_{k-1}$. But by the same reasoning, because $\gamma'(\gamma_{k-1}) R \gamma'(\gamma_{0})$, $j_{0} > j_{k-1}$ (noting $0 = (k-1)+1 \mathrm{\: mod \:} k$), a contradiction. 

\end{proof}

Note that we only need $H \notin \mathcal{H}$, not $G \notin \mathcal{H}$, in the above, but $G \notin \mathcal{H}$ will be true in the instances where we use this proposition. Informally, we have shown that, for a graph morphism $g: H \to G$ with $G, H \notin \mathcal{H}$, where $\mathcal{H}$ is an $\mathrm{NSOP}_{n}$ hereditary class of directed graphs, we can pull $g$ back to a graph morphism that \textit{removes} any cycles lying above a minimal-length cycle in $G$. The next, short, lemma says that if we do this, we will not \textit{create} a cycle lying above a minimal-length cycle in $G$, if there were none before.

\begin{lemma}\label{pulling back does not create cycles}

Let $g: H \to G$ be a graph homomorphism, and let $\gamma \hookrightarrow G$ be an $n$-cycle in $G$. Suppose there is no $n$-cycle $\gamma' \hookrightarrow H$ such that $g \circ \gamma' = \gamma$. Let $h: H' \to H$ be another graph homomorphism. Then there is no $n$-cycle $\gamma'' \hookrightarrow H'$ such that $(g \circ h) \circ \gamma'' = \gamma$.

\end{lemma}

\begin{proof}
    Suppose such a cycle $\gamma'' \hookrightarrow H'$ did exist. Define $\gamma' := h \circ \gamma''$. Then $\gamma' \hookrightarrow H$ would satisfy $g \circ \gamma'= g \circ ( h \circ \gamma'')  = (g \circ h) \circ \gamma'' = \gamma$, a contradiction.
\end{proof}

We are now ready to prove that $\mathfrak{o}(\mathcal{H})$ is an integer when $\mathcal{H}$ is defined by a finite family of forbidden weakly embedded substructures, Theorem \ref{integerality}:

\begin{proof}
    Suppose $\mathcal{H}$ is defined by a finite family of forbidden weakly embedded substructures, and $\mathfrak{o}(\mathcal{H})$ is not an integer. By Lemma \ref{reduction to directed graph defined by a finite family of forbidden weakly embedded substructures} and the proof of Lemma \ref{finitary and infinitary conditions}, we may assume that, for $n \geq 3$ some integer, $\mathcal{H}$ is an $\mathrm{NSOP}_{n}$ hereditary class of \textit{directed graphs} with the graph relation $R$, defined by a finite family of forbidden weakly embedded substructures, such that $H:=\{a_{i}\}_{i < \omega} \in \mathcal{H}$ where $a_{i} R a_{j}$ exactly when $i < j$, but, for some integers $N, k > 0$ with $n-1 <  \frac{N}{k} < n$, $G := \{b_{i}\}^{N -1}_{i =1} \notin \mathcal{H}$ where, for all $i, j$ such that $0 \leq i <N$ and $1 \leq j \leq k$,  $b_{i}R b_{i+j \mathrm{\: mod \:} N}$ (i.e., $G$ is an $(N, k)$-cycle for the relation $R$, in the sense of Notation \ref{(N, k)-cycle}.) We make the following fundamental observation about this graph $G$:

    \begin{claim}\label{(N, k)-cycle has no m-cycles}
       The graph $G$ has no directed $m$-cycles for $m < n$.
    \end{claim}

    \begin{proof}(of claim)

    Suppose otherwise, and let $b_{i_{0}}, \ldots b_{i_{m-1}}$ be an $m$-cycle. We may assume $m$ is minimal, so by Claim \ref{minimal length cycles}, none of the $i_{j}$ are repeated. By the definition of $G$, the map $b_{i} \mapsto b_{i - i_{0} \mathrm{\:mod\:} N}$ is a graph isomorphism, so we may assume $i_{0} = 0$. 
    
    We first show inductively that $i_{j} \leq kj$ for all $j \leq m -1$; suppose we have shown this for $j < m-1$ and we show it for $j+1$. First, since $\frac{N}{k} >  n-1 \geq m$, $mk < N$, so because $j \leq m-1$, $i_{j} \leq kj \leq km-k < N-k$. So by the definition of $G$ and since $b_{i_{0}}, \ldots b_{i_{m-1}}$ is an $m$-cycle, $k \geq i_{j+1} - i_{j} > 0$. So because $i_{j} \leq kj$, $i_{j+1} \leq kj + k = k(j+1)$.

So $i_{m-1} \leq k(m-1) = km - k$. But $km < N$, so $i_{m-1} < N -k $. Then by definition of $G$, and $i_{0} = 0$, it is impossible that $b_{i_{m-1}}Rb_{i_{0}}$, contradicting that $b_{i_{0}}, \ldots b_{i_{m-1}}$ is an $m$-cycle.
    \end{proof}

The proof of next claim contains the main argument of the proof of the theorem.

\begin{claim}\label{arbitrarily large minimal cycles}

For all integers $m \geq n$, there is a graph $K \notin \mathcal{H}$ none of whose directed cycles have length less than $m$.
    
\end{claim}

\begin{proof}(of claim)

In the base case $m = n$, we can take $K = G$, by Claim \ref{(N, k)-cycle has no m-cycles}. It remains to prove the inductive step: suppose there is a graph $K \notin \mathcal{H}$ where the smallest directed cycle has length $m \geq n$. We find a graph $K' \notin \mathcal{H}$ where the smallest directed cycle has length greater than $m$. Let $\gamma_{1} \hookrightarrow K, \ldots \gamma_{N^{*}} \hookrightarrow K$ enumerate the directed cycles of minimal length (i.e., of length $m$) in $K$. By induction on $0 \leq j \leq N^{*}$, we find some graph $K_{j} \notin \mathcal{H}$ with a graph homomorphism $g_{j}: K_{j} \to K$ for $K_{j} \notin \mathcal{H}$ such that, for $1 \leq i \leq j$, there is no $m$-cycle $\gamma \hookrightarrow K_{j}$ such that $\gamma_{i} = g_{j} \circ \gamma$. For $j = 0$ we can just take $g_{0}: K_{0} \to K$ to be the identity on $K_{0} := K$, because $K_{0} \notin \mathcal{H}$ and there are no additional requirements on $g_{0}$. Now suppose, for $0 \leq j < N^{*}$, we have found $g_{j}: K_{j} \to K$ as desired; we find $g_{j+1}: K_{j+1} \to K$ as desired. Recall that $\mathcal{H}$ is $\mathrm{NSOP}_{n}$, so $\mathrm{NSOP}_{m}$ because $m \geq n$. So by Lemma \ref{cycle removal}, we can find an $m$-helix map $h_{j+1}: \tilde{K_{j}} \twoheadrightarrow K_{j}$ of some length, such that $\tilde{K_{j}} \notin \mathcal{H}$ and there is no $m$-cycle $\gamma \hookrightarrow K_{j}$ such that $\gamma_{j+1} = (g_{j} \circ h_{j+1}) \circ \gamma$. Define $K_{j+1}:= \tilde{K_{j}}$, so $K_{j+1} \notin \mathcal{H}$, and $g_{j+1}:= g_{j} \circ h_{j+1} $; then there is no $m$-cycle $\gamma \hookrightarrow K_{j+1}$ such that $\gamma_{j+1} = (g_{j+1} ) \circ \gamma$. Moreover, by Lemma \ref{pulling back does not create cycles} and the hypothesis on $g_{j}$, because $g_{j+1}= g_{j} \circ h_{j+1} $, for $i \leq j$ there is no $m$-cycle $\gamma \hookrightarrow K_{j+1}$ such that $\gamma_{i} = (g_{j+1} ) \circ \gamma$, so this is true for $i \leq j+1$, completing the induction.

Then, define $K':= K_{N^{*}}$, and $g: K' \to K$ to be $g_{N^{*}}$. Then $K'$ has no $m'$-cycles for $m' < m$, because for such a cycle $\gamma \hookrightarrow K'$, $g \circ \gamma$ would be an $m'$-cycle in $K$, when we assumed that $K$ has no $m'$-cycles. But $K'$ does not even have any $m$-cycles, because for an $m$-cycle $\gamma \hookrightarrow K'$, $g \circ \gamma$ would have to appear among the $\gamma_{1} \hookrightarrow K, \ldots \gamma_{N^{*}} \hookrightarrow K$, contradicting the hypothesis on $g_{N^{*}}$.
    
\end{proof}

We now complete the proof of the theorem. This last step relies on a similar argument to one used in the discussion preceding Conjecture 5.2 of \cite{LZ17}: if a hereditary class of graphs, defined by a finite family of forbidden weakly embedded subgraphs, contains all cyclefree graphs, then for some $m$ it must contain all graphs without a cycle of length less than $m$. We give the full argument as it applies to our setting. For $\mathcal{F}$ a finite family of graphs such that $\mathcal{H} = \mathcal{H}(\mathcal{F})$, let $m$ be such that every graph in $\mathcal{F}$ with a directed cycle has a directed cycle of length less than $m$. By the last claim, let $K \notin \mathcal{H}$ have no directed cycles of length less than $m$. Then $\mathcal{H}$ must contain some $F \in \mathcal{F}$, which then has no directed cycles of length less than $m$, so no directed cycles at all. Of course, $F \notin \mathcal{H}$. But:

\begin{claim}\label{every cycle-free graph embeds in the linear order}

Every cycle-free directed graph $F$ can be weakly embedded in the graph $H$ (equivalently, in the linear order relation on $\
\omega$).

\end{claim}

\begin{proof}
    The following is well-known. Let $F$ be a finite, cycle-free directed graph. It suffices to show that, by adding edges between pairs of distinct vertices that do not currently have an edge in either direction, we can extend the edge relation on $F$ to a linear order on the vertices.
   
   If $F$ does not have an edge (in either direction) between every pair of (distinct) vertices, we show we can add an edge between an arbitrary pair of vertices that currently have no edge, and still get a cycle-free graph; then by induction, we will get a cycle-free graph such that every pair of vertices has an edge, which must be a linear order, because otherwise it would have a $3$-cycle. Let $v_{1}, v_{2}$ be a pair of vertices with no edge. Since $F$ has no cycles, either there is no (directed) path from $v_{1}$ to $v_{2}$, or there is no path from $v_{2}$ to $v_{1}$. In the first case, adding an edge from $v_{2}$ to $v_{1}$ will preserve that $F$ is cycle-free, so we are done. Similarly, in the second case, adding an edge from $v_{1}$ to $v_{2}$ will preserve that $F$ is cycle-free, so we are done in this case as well.
\end{proof}

So $F\notin \mathcal{H}$ weakly embeds in $H$, and $H \notin \mathcal{H}$, a contradiction.

\end{proof}

\section{Interval helices and the general case}

We return to Question \ref{hereditary class distinctness} for general hereditary classes. By Fact \ref{closure under helix maps} and Remark \ref{reduction to directed graphs}, the following statements are equivalent:

(1) the properties $\mathrm{NSOP}_{n}$ for integers $n \geq 3$ and the properties $\mathrm{NSOP}_{r}$ for real numbers $r \geq 3$ are distinct for hereditary classes (i.e., the answer to the second part of Question 2.14 is yes).

(2) For some real number $r \geq 3$ and $n = \lceil r \rceil$ the next integer, there is a hereditary class $\mathcal{H}$ of directed graphs with graph relation $R$, closed under weak embeddings and closed under $n$-helix maps, such that the graph $\{a_{i}\}_{i < \omega} \in \mathcal{H}$ where $a_{i} R a_{j}$ exactly when $i < j$, but $\{a_{\theta}\}_{\theta \in S^{1}} \notin \mathcal{H}$ where $a_{\theta} R a_{\psi}$ exactly when $\psi$ lies at most $\frac{2\pi}{r}$ radians counterclockwise from $\theta$.

In proving Theorem \ref{integerality}, we showed that (2) fails when $\mathcal{H}$ is also required to be defined by a finite family of forbidden weakly embedded substructures: in this case, if $\{a_{i}\}_{i < \omega} \in \mathcal{H}$ but $\{a_{\theta}\}_{\theta \in S^{1}} \notin \mathcal{H}$ as in the statement, there must be some $n$-helix map under which $\mathcal{H}$ is not closed. But we showed more: $\mathcal{H}$ must fail to be closed under a special kind of $k$-helix map for some integer $k \geq n$, specifically those associated with a directed cyclic $k$-decomposition of the form $H = \bigsqcup^{k-1}_{i =0} G_{i} \sqcup D$, where the $G_{i}$ are independent sets. Unwinding the proof, these are the helix maps constructed in the proof of Proposition \ref{cycle removal}. There, the cyclic $k$-decompositions are directed, and the $G_{i}$ are independent sets, because the $G_{i}$ are the fibers of a graph homomorphism above vertices $v_{i}$, where $v_{i} R v_{i + 1 \mathrm{\:mod \:} }$ for $i < n$. To refer to special helix maps like these, we give the following three increasingly strong definitions:

\begin{definition}\label{special helices}
    Let $h: H \twoheadrightarrow G $ be an $n$-helix map associated with a cyclic $n$-decomposition $G = \bigsqcup^{n-1}_{i =0} G_{i} \sqcup D$. Then:

    \begin{itemize}
        \item The map $h$ is a \textit{directed $n$-helix map} if the cyclic $n$-decomposition $G = \bigsqcup^{n-1}_{i =0} G_{i} \sqcup D$ is directed (Definition \ref{cyclic n-decomposition}).

        \item The map $h$ is an \textit{$n$-interval helix map} if it is an directed $n$-helix map where the $G_{i}=\sqcup_{j < n_{i}} S^{j}_{i}$ are disjoint unions of sets $S^{j}_{i}$ that are linearly ordered by $R$, where there are no edges between vertices of $S^{j}_{i}$ and $S^{j'}_{i}$ (in either direction) when $j \neq j'$.

\item The map $h$ is an \textit{$n$-anticlique helix map} if it is an directed $n$-helix map where the $G_{i}$ are independent sets (and is therefore an $n$-interval helix map.)

    \end{itemize}

    For $\ell > 1$, \emph{directed $n$-helix maps of length $\ell$}, etc., are as in Definition \ref{helix map}.
\end{definition}

It is therefore a corollary of the proof of Theorem 3.1 that the following stronger statement holds:

\begin{cor}\label{anticlique helix integrality}
 Let $\mathcal{H}$ be a hereditary class of directed graphs with graph relation $R$, defined by a finite family of forbidden weakly embedded substructures. Suppose that for $r \geq 3$ a real number, the graph $\{a_{i}\}_{i < \omega} \in \mathcal{H}$ where $a_{i} R a_{j}$ exactly when $i < j$, but $\{a_{\theta}\}_{\theta \in S^{1}} \notin \mathcal{H}$ where $a_{\theta} R a_{\psi}$ exactly when $\psi$ lies at most $\frac{2\pi}{r}$ radians counterclockwise from $\theta$. Then $\mathcal{H}$ is not closed under some $k$-anticlique helix map, for some integer $k \geq n$.
\end{cor}

    Though, as the above corollary makes precise, anticlique helix maps are enough to prove that $\mathfrak{o}(\mathcal{H})$ is an integer when $\mathcal{H}$ is a hereditary class defined by a finite family of forbidden weakly embedded substructures, in the main results of this section we will be interested in the larger class of interval helix maps, obtaining a stronger theorem than were we just to consider the anticlique helix maps. Because all anticlique helix maps are interval helix maps, Corollary \ref{anticlique helix integrality} \textit{applies to the class of interval helix maps as well}, and interval helix maps are similarly enough to prove that $\mathfrak{o}(\mathcal{H})$ is an integer when $\mathcal{H}$ is a hereditary class defined by a finite family of forbidden weakly embedded substructures. But moreover, interval helix maps are enough to show that the properties $\mathrm{NSOP}_{n}$, for integers $n \geq 3$, can in fact be restated so that they extend to the hierarchy $\mathrm{NSOP}_{r}$, for real $r \geq 3$, as in Observation \ref{restatement of NSOP_n}. More precisely, the following follows from the proofs of Observation \ref{restatement of NSOP_n} (and Lemma \ref{finitary and infinitary conditions}):

    \begin{cor}\label{interval helix restatement of NSOP_n}
        Let $R$ be a hereditary class of graphs closed under weak embeddings, with graph relation $R$, and suppose  the graph $\{a_{i}\}_{i < \omega} \in \mathcal{H}$ where $a_{i} R a_{j}$ exactly when $i < j$. Suppose additionally that $\mathcal{H}$ is closed under $n$-interval helix maps. Then the graph $\{a_{\theta}\}_{\theta \in S^{1}} \in \mathcal{H}$ where $a_{\theta} R a_{\psi}$ exactly when $\psi$ lies at most $\frac{2\pi}{n}$ radians counterclockwise from $\theta$. 
    \end{cor}

    It is not clear that the same is true when $\mathcal{H}$ is closed under $k$-anticlique helix maps for integers $k \geq n$. This motivates us to consider interval helix maps, not just anticlique helix maps, in Theorem \ref{interval helix non-integrality} below.

    \begin{remark}\label{alternative motivation for being closed under special helix maps}
        We can motivate the property of being closed under interval helix maps without directly appealing to helix maps themselves. It follows from the proof of Fact \ref{closure under helix maps} (noting that the property of being a disjoint union of sets that are linearly ordered by $R$, where there are no edges between this sets, is itself preserved under taking disjoint unions of graphs with no new edges), that

        \begin{fact}\label{alternative definition of being closed under interval helix maps}

For $\mathcal{H}$ a hereditary class of graphs with edge relation $R$. Then $\mathcal{H}$ is closed under $n$-anticlique helix maps if and only if, for every graph of the form $AB$ with a sequence $\{A_{i}\}_{i < \omega} \in \mathcal{H}$ such that

\begin{itemize}
    \item $A_{i}A_{j} \equiv AB$ for all $i < j$,

    \item for $C = A_{0} \cap A_{1}$, and for all $i < j$, the only edges from any $v \in A_{i} \backslash C$ to any $v' \in A_{j} \backslash C$ are in the direction from $v$ to $v'$,

    \item and for all $i < \omega$, $A_{i} \backslash C =\sqcup_{j < n_{i}} S^{j}_{i}$ is a disjoint union of sets $S^{j}_{i}$ that are linearly ordered by $R$, where there are no edges between vertices of $S^{j}_{i}$ and $S^{j'}_{i}$ when $j \neq j'$,

\end{itemize}

then there are $\{A_{i}\}^{n-1}_{i = 0}$ such that $A_{i}A_{i+1 \mathrm{\: mod \:} n} \equiv AB$ for $i < n$.

        \end{fact}

Similar equivalences hold for being closed under directed helix maps, and being closed under anticlique helix maps. Note the similarity of these equivalent conditions to the definition of $\mathrm{NSOP}_{n}$: they say that if a quantifier-free type $p(X, Y)$ has $\{A_{i}\}_{i< \omega} \in \mathcal{H}$ such that $\models p(A_{i}, A_{j})$ for $i < j$, then $p(X, Y)$ also has an $n$-cycle, assuming that $\{A_{i}\}_{i< \omega} \in \mathcal{H}$ satisfies additional conditions. However, though these equivalent conditions for, say, being closed under $n$-interval helix maps appear to be refinements of $\mathrm{NSOP}_{n}$, it should not be expected to follow that being closed under $n$-interval helix maps implies being closed under $m$-interval helix maps for $n < m$. (However, being closed under \textit{directed} $n$-helix maps does imply being closed under directed $m$-helix maps for $m > n$.)
    \end{remark}

As stated at the beginning of this section, if (2) were false--i.e., there is a hereditary class (closed under weak embeddings) of graphs whose graph relation exhibits $\mathrm{SOP}_{r}$ that is closed under $n$-helix maps, where $n = \lceil r \rceil$ is the next integer after $r$--then the real-valued and integer-valued hierarchies would be distinct at the level of hereditary classes. Toward showing this conjectural claim, the goal of this section will be to prove that there is a hereditary class of graphs closed under weak embeddings, whose graph relation exhibits $\mathrm{SOP}_{r}$, that is closed under $n$-interval helix maps for all integers $n > r$. Informally speaking, we will have shown that, while interval helix maps are powerful enough for showing that $\mathfrak{o}(\mathcal{H})$ is an integer for hereditary classes defined by a finite family of forbidden weakly embedded substructures (as in Corollary \ref{anticlique helix integrality} to the proof of Theorem \ref{integerality}), and for showing that the integer-valued $\mathrm{NSOP}_{n}$ hierarchy extends to the real-valued $\mathrm{NSOP}_{r}$ hierarchy (as in Corollary \ref{interval helix restatement of NSOP_n} to the proof of Observation \ref{restatement of NSOP_n}), it is not powerful enough for showing that the hierarchies are the same, or that $\mathfrak{o}(\mathcal{H})$ is an integer, for general hereditary classes.

\begin{theorem}\label{interval helix non-integrality}
    Let $r > 3 $ be a real number that is not an integer. Then there is a hereditary class $\mathcal{H}$ of graphs with graph relation $R$ that is closed under weak embeddings, such that  $\{a_{i}\}_{i < \omega} \in \mathcal{H}$ where $a_{i} R a_{j}$ exactly when $i < j$, but $\{a_{\theta}\}_{\theta \in S^{1}} \notin \mathcal{H}$ where $a_{\theta} R a_{\psi}$ exactly when $\psi$ lies at most $\frac{2\pi}{r}$ radians counterclockwise from $\theta$, and such that $\mathcal{H}$ is closed under $n$-interval helix maps for any integer $n > r$.
\end{theorem}

\begin{remark}\label{real-valued helix maps}
    We could have easily extended the definitions of $n$-helix maps, $n$-interval helix maps, etc. to non-integer values of $n$; then we expect the proof of Theorem \ref{anticlique helix integrality} to extend to these maps as well, yielding a hereditary class closed under $r'$-helix maps for real numbers $r' >r$. (We will point out a step of the proof that it is important to note extends to $r'$-helix maps below.) For ease of exposition, we will only explicitly consider $n$-interval helix maps for integers $n > r$, since the proof should straightforwardly extend to this more general case. But the extension to $r'$-helix maps for reals $r' > r$, informally speaking, will tell us that interval helices are not even powerful enough to show that $\mathfrak{o}(\mathcal{H})$ cannot take the value $r$, for $r$ some arbitrary real value greater than $3$. (It will therefore tell us, in a stronger sense than in the statement of Theorem \ref{interval helix non-integrality}, that interval helices are not enough to show that $\mathfrak{o}(\mathcal{H})$ is not an integer.)
\end{remark}

To prove this theorem, we will define a property of graphs that holds for subgraphs of $\{a_{\theta}\}_{\theta \in S^{1}}$--specifically, for $(N, k)$-cycles for $r \leq \frac{N}{k} < n^{*}$, where $n^{*}:= \lceil r \rceil$ is the next integer after $r$. We want the class defined by \textit{omitting} graphs with this property to be closed under $n$-interval helix maps for $n > r$, and to contain all cycle-free graphs.

\begin{definition}\label{cyclically n-indecomposable}
    Let $n \geq 3$ be an integer. A graph $G$ with $|G| \geq 2$ is \textit{cyclically $n$-indecomposable} if it has no full, directed cyclic $n$-decomposition $G = \sqcup^{n-1}_{i=0} G_{i}$ where at least two of the $G_{i}$ are nonempty.
\end{definition}

Choose any positive integers $(N, k)$ such that $r \leq \frac{N}{k} < n^{*}$. Let $G$ be the $(\textit{N, k})$-cycle graph as in Notation \ref{(N, k)-cycle} and the proof of Theorem \ref{integerality}: $G := \{b_{i}\}^{N -1}_{i =0} \notin \mathcal{H}$ where, for all $i, j$ such that $0 \leq i <N$ and $1 \leq j \leq k$,  $b_{i}R b_{i+j \mathrm{\: mod \:} N}$. We show:

\begin{lemma}\label{(N, k)-cycle cyclically n^{*}-indecomposable}
    The graph $G$ as above is cyclically $n^{*}$-indecomposable.
\end{lemma}

\begin{proof}
 Suppose otherwise, and let $G = \sqcup^{n^{*}-1}_{i = 0} G_{i}$ be a full, directed cyclic $n^{*}$-decomposition of $G$. Let $f: G \to \{0, \ldots n^{*}-1\}$ send $v \in G_{i}$ to $i$. Because $G = \sqcup^{n^{*}-1}_{i = 0} G_{i}$ is a directed cyclic $n$-decomposition, and $b_{i} R b_{i+1} \mathrm{\:mod\:} n^{*}$ for $0 \leq i \leq n-1$, either

 (a) $f(b_{i+1 \mathrm{\: mod \:} n^{*}})-f(b_{i})= 1 \mathrm{\: mod \:} n$, or

 (b) $f(b_{i+1 \mathrm{\: mod \:} n^{*}}) = f(b_{i})$.

For $1 \leq i \leq N-1$, define $\delta_{i}$ to be equal to $1$ if (a) holds, and equal to $0$ if (b) holds. We claim that, for $N' < N-1$, $\sum^{N'}_{i=0} \delta_{i} \leq \frac{N'}{k} +1$: if $\delta_{i} = 1$, then because $G = \sqcup^{n^{*}-1}_{i = 0} G_{i}$ is a directed cyclic $n$-decomposition and $b_{i}Rb_{j+1}$, $b_{i+1} R b_{j+1}$ for each $j < N-1$ with $i < j < i+k$, for each such $j$, $\delta_{j} = 0$. So, enumerating the values $i \leq N'$ such that $\delta_{i} = 1$ as $i_{0} < \ldots < i_{m} \leq N'$, $ mk \leq i_{m} \leq N'$, so $m \leq \frac{N'}{k}$ and $\sum^{N'}_{i=0} \delta_{i} \leq \frac{N'}{k} +1$. Since more than one of the $G_{i}$ are nonempty, we may additionally assume, by symmetry, that $\delta_{0} = 1$, and that $f(b_{0}) = 0$. Now set $N'= N-k$. Then $\sum^{N-k}_{i=0} \delta_{i} \leq \frac{N-k}{k} +1 = \frac{N}{k} < n^{*}$, so $f(b_{1}) = 1$ while $1 \leq f(b_{N-k+1}) \leq n^{*} -1$. But because $G = \sqcup^{n^{*}-1}_{i = 0} G_{i}$ is a directed cyclic $n$-decomposition, $1 \leq f(b_{N-k+1}) \leq n^{*} -1$, $f(b_{0}) = 0$ and $b_{N-k+1}Rb_{0}$ imply that $f(b_{N-k+1}) = n^{*} -1 $, while $b_{N-k+1}Rb_{1}$ and $f(b_{1}) = 1$ imply that $f(b_{N-k+1})\neq n^{*} -1 $, a contradiction.

\end{proof}

To show that the class defined by omitting cyclically $n^{*}$-indecomposable graphs contains all cycle-free graphs, we prove the following short lemma:

\begin{lemma}\label{cyclically n-indecomposable graphs have cycles}
    Any graph that is cyclically $n$-decomposable for an integer $n \geq 3$ contains a cycle.
\end{lemma}

\begin{proof}
    Any finite graph $H$ without a cycle has a vertex $v$ such that there is no $v' \in H$ with $v'Rv$. Then $H =\{v\} \sqcup (H \backslash \{v\})$ gives a cyclic $n$-decomposition of $H$.
\end{proof}

We are now ready to prove Theorem \ref{interval helix non-integrality}.

\begin{proof}
    Let $\mathcal{H}$ be the class of graphs that omit all cyclically $n^{*}$-indecomposable graphs.  By the definition of being cyclically $n^{*}$-indecomposable, $\mathcal{H}$ is a class of graphs closed under weak embeddings; we show that $\mathcal{H}$ is as desired. By Lemma \ref{(N, k)-cycle cyclically n^{*}-indecomposable} and the proof of Lemma \ref{finitary and infinitary conditions}, $\{a_{\theta}\}_{\theta \in S^{1}} \notin \mathcal{H}$ where $a_{\theta} R a_{\psi}$ exactly when $\psi$ lies at most $\frac{2\pi}{r}$ radians counterclockwise from $\theta$. However, by Lemma \ref{cyclically n-indecomposable graphs have cycles}, $\{a_{i}\}_{i < \omega} \in \mathcal{H}$ where $a_{i} R a_{j}$ exactly when $i < j$, because $\{a_{i}\}_{i < \omega}$ has no cycles. It remains to show that $\mathcal{H}$ is closed under $n$-interval helix maps when $n > r$ is an integer. For this, it suffices to show that, if $H$ contains a cyclically $n^{*}$-indecomposable subgraph and $h: \tilde{H} \twoheadrightarrow H$ is an $n$-interval helix map, then $\tilde{H}$ contains a cyclically $n^{*}$-indecomposable subgraph. For $H' \subseteq H$ a cyclically $n^{*}$-indecomposable subgraph, $h|_{h^{-1}(H')}: h^{-1}(H') \twoheadrightarrow H'$ is also an $n$-interval helix map, so, replacing $H$ with $H'$, we can assume $H$ is cyclically $n^{*}$-indecomposable.

Suppose $h: \tilde{H} \to H$ is the $n$-interval helix map associated with the directed cyclic $n$-decomposition $H= \bigsqcup^{n-1}_{i =0} G_{i} \sqcup D$, for $G_{i}=\sqcup_{j < n_{i}} S^{j}_{i}$ disjoint unions of sets $S^{j}_{i}$ that are linearly ordered by $R$, where there are no edges between vertices of $S^{j}_{i}$ and $S^{j'}_{i}$ when $j \neq j'$. It suffices to show that, for sufficiently large $\ell$, and $h^{\ell}: \tilde{H}^{\ell} \to H$ the $n$-interval helix map of length $\ell$ associated with this cylcic $n$-decomposition, $\tilde{H}^{\ell}$ contains a cyclcially $n^{*}$-indecomposable subgraph. In the degenerate case where $D$ is empty and only one of the $G_{i}$ for $i < n$ is nonempty, we are done; otherwise, we prove the following preliminary claim, which we will use later:

\begin{claim}\label{nonempty base}
    $D$ is nonempty.
\end{claim}

\begin{proof}
    Otherwise, $H= \bigsqcup^{n-1}_{i =0} G_{i}$; we may assume $G_{0}$ is nonempty. But then, since $H= \bigsqcup^{n-1}_{i =0} G_{i}$ is a (full) directed cyclic $n$-decomposition of $H$, $H= \bigsqcup^{n^{*}-1}_{i =0} G'_{i}$, with $G'_{i} = G_{i}$ for $0 \leq i \leq n^{*}-2$ and $G'_{n^{*}-1}= \bigsqcup^{n}_{i=n^{*}-1} G_{i}$, is a full directed cyclic $n^{*}$-decomposition of $H$, where at least two of the $G'_{i}$ are nonempty. This contradicts that $H$ is cyclically $n^{*}$-indecomposable.
\end{proof}

Let $\ell_{1}= |H \backslash D|$, and choose $\ell_{2} $ to be larger than the number of partitions of $H$ into $n^{*}$ many sets; then let $\ell$ be equal to $\ell_{2} + 2\ell_{1}$. Writing $\tilde{H}^{\ell-1} = \bigsqcup^{\ell}_{j =0}\bigsqcup^{n-1}_{i = 0} G^{j}_{i} \sqcup D$ with notation as in the proof of Fact \ref{closure under helix maps}, for $j \leq \ell$ define $G^{j} := \bigsqcup^{n-1}_{i = 0} G^{j}_{i}$. Then let $I_{1}$ be the interval of integers $[0, \ell_{1} -1]$ of size $\ell_{1}$, let $I_{2}$ be the interval of integers $[\ell_{1}, \ell_{1} + \ell_{2} -1]$ of size $\ell_{2}$, and let $I_{3}$ be the interval of integers $[\ell_{1}+\ell_{2}, \ell-1]$ of size $\ell_{1}$. For $i = 1, 2, 3$ let $G^{I_{i}} : = \bigsqcup_{j \in I_{i}} G^{j}$.

For $v \in H \backslash D$, let $S(v)$ be the unique linearly ordered set $S^{j}_{i}$ (for $i < n$, $j < n_{i}$) such that $v \in S^{j}_{i}$ We show the following claim:

\begin{claim}\label{graph filtration}
There exists a filtration $D = K_{0} \subsetneq \ldots \subsetneq K_{s} = H$ of $H$ for $0 \leq i \leq s$, and $v_{i} \in K_{i}$ for $1 \leq i \leq s$, such that for $1 \leq i \leq s$,

(1) $K_{i}=K_{i-1}\cup\{v_{i}\} \cup \{v\in S(v_{i}): vRv_{i}\}$, and

(2) there exists some $v \in K_{i-1}$ such that $v_{i} R v$.

\end{claim}

\begin{proof}(of claim)

Suppose by induction we have found  $D = K_{0} \subsetneq \ldots \subsetneq K_{i} \subsetneq H$ satisfying (1) and (2). Then there must be some $v_{i+1} \in H \backslash K_{i}$, $v \in K_{i}$ such that $v_{i+1} R v$; otherwise, the only edges between any $w \in H \backslash K_{i}$ and $w' \in K_{i}$ go in the direction from $w'$ to $w$, and we get a full directed cyclic $n^{*}$-decomposition $H = K_{i} \sqcup H \backslash K_{i}$. But $H \backslash K_{i}$ is nonempty by assumption, and $K_{i} \supseteq D$ is nonempty by \ref{nonempty base}, contradicting cyclic $n^{*}$-indecomposability of $H$.  So $v_{i+1}$, $K_{i+1}:=K_{i}\cup\{v_{i}\} \cup \{v\in S(v_{i}): vRv_{i}\}$ are as desired, and continuing in this way we find $D = K_{0} \subsetneq \ldots \subsetneq K_{s} = H$ as desired.

\end{proof}

Similarly, we can find a filtration $D = K_{0} \subsetneq \ldots \subsetneq K_{-t} = H$  with $v_{-i} \in K_{-i}$ for $0 \leq i \leq t$, such that for $1 \leq i \leq t$, 

(1) $K_{-i}=K_{-(i-1)}\cup\{v_{-i}\} \cup \{v\in S(v_{-i}): v_{-i}R v \}$, and

(2) there exists some $v \in K_{-(i-1)}$ such that $vRv_{-i}$.

We now find a cyclically $n^{*}$-indecomposable subset of $\tilde{H}^{\ell}= G^{I_{1}} \sqcup G^{I_{2}} \sqcup G^{I_{3}} \sqcup D$. Noting that for $s$, $t$ as in the claim, $s, t \leq \ell_{1}$, define $G_{a} \subset G^{I_{1}}$, $G_{b} \subset G^{I_{2}}$, $G_{c} \subset G^{I_{3}}$ as follows:

$$G^{a} := \bigsqcup^{t}_{i=1} (G^{i-1}\cap (h^{\ell})^{-1}(K_{-i})) \sqcup \bigsqcup^{\ell_{1}}_{i=t+1}  G^{i}$$

$$G^{b} : = G^{I_{2}}$$

$$G^{c} := \bigsqcup^{s}_{i=1} (G^{\ell-i}\cap (h^{\ell})^{-1}(K_{i})) \sqcup \bigsqcup^{\ell_{1}}_{i=s+1}  G^{\ell - i} $$

Define $G^{*} = G^{a} \sqcup G^{b} \sqcup G^{c}\sqcup D$. The resulting set resembles a parallelogram as in Figure \ref{parallelogram} below.

\begin{figure}[hbt!]

{%
\begin{circuitikz}
\tikzstyle{every node}=[font=\normalsize]
\draw [dashed] (2.5,11) -- (11.75,11);
\draw [dashed] (2.5,8) -- (11.75,8);
\draw [dashed] (2.5,11) -- (2.5,8);
\draw [dashed] (11.75,11) -- (11.75,8);
\draw (2.5,8) to[short] (5.5,11);
\draw (5.5,11) to[short] (11.75,11);
\draw (11.75,11) to[short] (8.75,8);
\draw (2.5,8) to[short] (8.75,8);
\draw [dashed] (5.5,11) -- (5.5,8);
\draw [dashed] (8.75,8) -- (8.75,11);
\node [font=\normalsize] at (4,7.25) {$I_{1}$};
\node [font=\normalsize] at (7,7.25) {$I_{2}$};
\node [font=\normalsize] at (10.25,7.25) {$I_3$};
\node [font=\normalsize] at (4.25,9.25) {$G^{a}$};
\node [font=\normalsize] at (7,9.25) {$G^b$};
\node [font=\normalsize] at (9.5,9.25) {$G^c$};
\end{circuitikz}
}%

\caption{Visual representation of the set $G^{*}\backslash D$ within $\tilde{H}^{\ell}$.}

\label{parallelogram}
\end{figure}
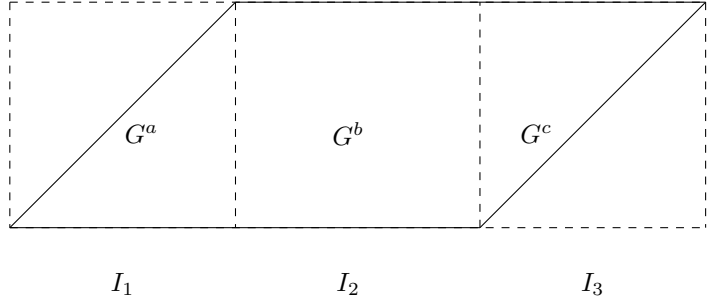

We show that $G^{*} \subseteq \tilde{H}^{\ell}$ is cyclically $n^{*}$-indecomposable. Suppose $G^{*}=:\bigsqcup^{n^{*}-1}_{i =0} (G^{*})^{i}$ is a full directed cyclic $n^{*}$-decomposition of $G^{*}$; as in the proof of Lemma \ref{(N, k)-cycle cyclically n^{*}-indecomposable}, define $f: G^{*} \to \{0, \ldots, n^{*}-1\}$ so that for $v \in G^{*}$, $f(v) =i$ when $v \in (G^{*})^{i}$. We show that $f$ is constant, a contradiction.

We prove the following claim:

\begin{claim}\label{monochromatic column}

There is $j^{*} \in I_{2}$ such that $f$ is constant on $G^{j^{*}} \sqcup D$.

\end{claim}

\begin{proof}(of claim)

    For each $j \in I_{2}$, $h_{j}^{\ell}:=h^{\ell}|_{G^{j} \sqcup D}: G^{j} \sqcup D \to H $ is a bijection. Since we ensured that $|I_{2}|$ is larger than the number of partitions of $H$ into $n^{*}$ sets, there must then, by the pigeonhole principle, be $j < j'$ in $I_{2}$ such that $f\circ (h^{\ell}_{j})^{-1}= f\circ (h^{\ell}_{j'})^{-1}$. Define $\tilde{f}:= f\circ (h^{\ell}_{j})^{-1}= f\circ (h^{\ell}_{j'})^{-1}$; we show $f$ must be constant. Suppose otherwise; form the partition $H := \bigsqcup^{n^{*}-1}_{i =0} F_{i}$ where $F_{i} = f^{-1}(i)$. Then at least two of the $F_{i}$ are nonempty because $f$ is nonconstant; we show $H := \bigsqcup^{n^{*}-1}_{i =0} F_{i}$ is a directed cyclic $n^{*}$-decomposition of $H$, which will then be full, obtaining a contradiction to cyclic $n^{*}$-indecomposability of $H$. Let $v, v' \in H$ satisfy $vRv'$; we show that either $\tilde{f}(v') = \tilde{f}(v)$ or $\tilde{f}(v') = \tilde{f}(v) + 1 \mathrm{\:mod\:} n^{*}$, as desired. For $H = \bigsqcup^{n}_{i=0} G_{i} \sqcup D$ the original cyclic $n$-decomposition of $H$ from which the helix map $h^{\ell}$ is obtained, there are two possibilities, by definition of a directed cyclic $n$-decomposition: either $v, v'$ belong to a single set of the form $G_{i} \sqcup D$, or $v \in G_{i}$ and $v' \in G_{i+1 \mathrm{\: mod \:} n}$. By construction of the helix map, in the first case, $(h^{\ell}_{j})^{-1}(v)R(h^{\ell}_{j})^{-1}(v')$, so either $f((h^{\ell}_{j})^{-1}(v'))=f((h^{\ell}_{j})^{-1}(v))$, in which case $\tilde{f}(v') = \tilde{f}(v)$, or $f((h^{\ell}_{j})^{-1}(v'))=f((h^{\ell}_{j})^{-1}(v))+1 \mathrm{\: mod \:} n^{*}$, in which case $\tilde{f}(v') = \tilde{f}(v) + 1 \mathrm{\:mod\:} n^{*}$. In the second case, from the construction of the helix map and that $j < j'$, we see that $(h_{j}^{\ell})^{-1}(v) R (h_{j'}^{\ell})^{-1}(v')$. So either $f((h^{\ell}_{j'})^{-1}(v'))=f((h^{\ell}_{j})^{-1}(v))$, in which case $\tilde{f}(v') = \tilde{f}(v)$, or $f((h^{\ell}_{j'})^{-1}(v'))=f((h^{\ell}_{j})^{-1}(v))+1 \mathrm{\: mod \:} n^{*}$, in which case $\tilde{f}(v') = \tilde{f}(v) + 1 \mathrm{\:mod\:} n^{*}$.
\end{proof}

Let $c^{*}$ be the constant value of $f$ on $G^{j^{*}} \sqcup D$. We show that $f$ takes the value $c^{*}$ on all of $G^{*}$. We show this on $G_{+}^{*}:= G^{*} \cap( \bigsqcup^{\ell-1}_{i = j^{*}} G^{j} \sqcup D)$, thereby showing the same for $G_{-}^{*}:= G^{*} \cap (\bigsqcup^{j^{*}}_{i = 0} G^{j} \sqcup D)$ by symmetry. Let $D = K_{0} \subsetneq \ldots \subsetneq K_{s} = H$ for $0 \leq i \leq s$, and $v_{i} \in K_{i}$ for $1 \leq i \leq s$, be as in Claim \ref{graph filtration}.  We know that $f$ takes the value $c^{*}$ on all of $K_{0} = D$; let us show, by induction on $1 \leq i \leq s$, that $f$ takes the value $c^{*}$ on all of $G^{*}_{+} \cap h^{-1}_{\ell}(K_{i} \backslash K_{i-1})$, which will suffice because $D = K_{0} \subset \ldots \subset K_{s} = H$. Suppose, by the induction hypothesis, that $f$ takes the value $c^{*}$ on all of $G^{*}_{+} \cap h^{-1}_{\ell}(K_{i-1})$; we show it takes the value $c^{*}$ on all of $G^{*}_{+} \cap h^{-1}_{\ell}(K_{i} \backslash K_{i-1})$. A vertex of $G^{*}_{+} \cap h^{-1}_{\ell}(K_{i} \backslash K_{i-1})$ will belong to $G^{j} \cap h^{-1}_{\ell}(K_{i} \backslash K_{i-1})$ for some $j$ with $j^{*} \leq j \leq \ell-i$, and we may assume $j^{*} < j$ since $f$ is already known to take value $c^{*}$ on all of $G^{j^{*}}$; it will then be enough to show that that $f$ takes the value $c^{*}$ on all of $G^{j} \cap h^{-1}_{\ell}(K_{i} \backslash K_{i-1})$. By condition (1) of Claim \ref{graph filtration}, $K_{i} \backslash K_{i-1} \subseteq \{v_{i}\} \cup \{v\in S(v_{i}): vRv_{i}\}$ and $\{v_{i}\} \cup \{v\in S(v_{i}): vRv_{i}\} \subset K_{i}$, so $G^{j} \cap h^{-1}_{\ell}(K_{i} \backslash K_{i-1}) \subseteq G^{j} \cap (h^{\ell})^{-1}(\{v_{i}\} \cup \{v\in S(v_{i}): vRv_{i}\}) \subseteq G^{*}_{+}$; thus it suffices to show that $f$ takes the value $c^{*}$ on $G^{j} \cap (h^{\ell})^{-1}(\{v_{i}\} \cup \{v\in S(v_{i}): vRv_{i}\})$, which will be equal to $G_{k_{i}}^{j} \cap (h^{\ell})^{-1}(\{v_{i}\} \cup \{v\in S(v_{i}): vRv_{i}\})$ where $\{v_{i}\} \cup \{v\in S(v_{i}): vRv_{i}\} \subseteq G_{k_{i}}$. By construction of the helix map, there is a unique vertex $\tilde{v}_{i} \in G_{k_{i}}^{j} \cap (h^{\ell})^{-1}(\{v_{i}\} \cup \{v\in S(v_{i}): vRv_{i}\})$ such that $h^{\ell}(\tilde{v}_{i}) = v_{i}$. By (2) of Claim \ref{graph filtration}, there is some $v \in K_{i-1}$ such that $v_{i} R v$. We prove the following claim about $v$:

\begin{claim}\label{pointing outside of component}
    The vertex $v$ does not belong to $G_{k_{i}}$ (so must belong to either $D$ or $G_{k_{i}+1 \mathrm{\:mod\:} n}$).
\end{claim}

\begin{proof}(of claim)

First, $v \notin S(v_{i})$: otherwise, by (1) of Claim \ref{graph filtration}, since $v \in K_{i-1}$, $v_{i} \in K_{i-1}$, contradicting the assumption that $v_{i} \in K_{i} \backslash K_{i-1}$. So, because $h^{\ell}$ is an $n$-interval helix map of length $\ell$, if $v$ is in $G_{k_{i}}$ it must be in one of the other $R$-linearly ordered sets partitioning $G_{k_{i}}$, none of whose vertices have edges from vertices of $S(v_{i})$. This contradicts $v_{i} R v$.

\end{proof}

Additionally, let $w$ be the $R$-least vertex of $S(v_{i})$. We prove the following about $w$:

\begin{claim}\label{vertex pointing into least vertex}
    There is a vertex $w' \in H \backslash G_{k_{i}}$ such that $w' R w$ (so either $w' \in D$ or $w' \in G_{k_{i}-1 \mathrm{\:mod\:} n}$).
\end{claim}

\begin{proof}(of claim)
As in the proof of claim \ref{cyclically n-indecomposable graphs have cycles}, there must be some $w' \in H$ such that $w' R w$; we show $w' \notin G_{k_{i}}$. First, $w' \notin S(v_{i})$ because $w$ is the $R$-least vertex in $S(v_{i})$. But $w'$ is not in any of the other disjoint $R$-linearly ordered sets in $G_{i}$, because there are no edges from vertices of those sets to vertices of $S(v_{i})$.

\end{proof}

Because $v \in K_{i-1}$, for the unique $\tilde{v} \in G^{\ell-(i-1)}\sqcup D$, either in $G^{\ell-(i-1)}_{k_{i}+1 \mathrm{\:mod\:} n}$ or $D$ by Claim \ref{pointing outside of component}, such that $h^{\ell}(\tilde{v}) = v$, $\tilde{v} \in G^{*}_{+} \cap h^{-1}_{\ell}(K_{i-1})$ by construction of $G^{c}$. Since $\tilde{v} \in G^{\ell-(i-1)}_{k_{i}+1 \mathrm{\:mod\:} n} \sqcup D$, $\tilde{v}_{i} \in G_{k_{i}}^{j}$ and $v_{i} R v$, and $j \leq \ell - i < \ell-(i-1)$, $\tilde{v}_{i} R \tilde{v}$. Let $\tilde{w'}$ be the unique vertex in  $G^{j^{*}} \sqcup D$, either in $G^{j^{*}}_{k_{i}-1 \mathrm{\:mod\:} n}$ or $D$, such that $h^{\ell}(\tilde{w'})= w'$ where $w'$ is as in Claim \ref{vertex pointing into least vertex}. Finally, let $\tilde{w}$ be the $R$-least vertex of $G_{k_{i}}^{j} \cap (h^{\ell})^{-1}(\{v_{i}\} \cup \{v\in S(v_{i}): vRv_{i}\}) \subseteq G^{+}$, so $h^{\ell}(\tilde{w}) = w$, and $\tilde{w'} R \tilde{w}$ because $\tilde{w'} \in G^{j^{*}}_{k_{i}-1 \mathrm{\:mod\:} n} \sqcup D$, $\tilde{w} \in G_{j}^{i}$, $w' R w$ and $j^{*} < j$. In addition to $\tilde{w'} R \tilde{w}$ and $\tilde{v_{i}} R \tilde{v}$, we have $\tilde{w}R \tilde{v_{i}}$. But $f(\tilde{w}')= c^{*}$ by Claim \ref{monochromatic column}, while $f(\tilde{v}) = c^{*}$ by $\tilde{v} \in G^{*}_{+} \cap h^{-1}_{\ell}(K_{i-1})$ and the induction hypothesis that $f$ takes the value $c^{*}$ on all of $G^{*}_{+} \cap h^{-1}_{\ell}(K_{i-1})$. Two properties of full directed cyclic $n^{*}$-decompositions for $n^{*} > 3$ (or even full directed cyclic ``$r$-decompositions" for $r > 3$; see Remark \ref{real-valued helix maps}) are

(A) If $aRb$, $bRc$, $cRd$ and $f(a) = f(d) = c^{*}$, then $f(b)=f(c) = c^{*}$, and

(B) If $aRb$, $bRc$, and $f(a) = f(c) = c^{*}$, then $f(b) = c^{*}$.

But by property (A), $\tilde{w'} R \tilde{w}$, $\tilde{w}R \tilde{v_{i}}$, $\tilde{v_{i}} R \tilde{v}$ and $f(\tilde{w'}) =f(\tilde{v}) = c^{*}$, we have $f(\tilde{w})=f(\tilde{v_{i}}) = c^{*}$. Then, since every vertex in the rest of the $R$-linearly ordered set $G_{k_{i}}^{j} \cap h^{-1}_{\ell}(K_{i} \backslash K_{i-1})$ lies $R$-between $\tilde{w}$ and $\tilde{v_{i}}$, by (B) $f$ takes the value $c^{*}$ on all of $G_{i}^{j} \cap h^{-1}_{\ell}(K_{i} \backslash K_{i-1})$, as desired.

\end{proof}

\subsection{General (directed) helix maps and cyclic indecomposability} 

One may wish to extend Theorem \ref{interval helix non-integrality} to directed $n$-helix maps, or $n$-helix maps in general: find a hereditary class of directed graphs closed under weak embeddings, whose graph relation exhibits $\mathrm{SOP}_{r}$, but which is closed under directed $n$-helix maps for integers $n > r$ (equivalently, for $n = \lceil r \rceil$), or which is closed under $n= \lceil r \rceil$-helix maps more generally (i.e., which is $\mathrm{NSOP}_{n}$.) As discussed at the beginning of this section, extending Theorem \ref{interval helix non-integrality} to general $n$-helix maps would amount to showing the real-valued $\mathrm{NSOP}_{r}$ and integer-valued $\mathrm{NSOP}_{n}$-hierarchy are distinct at the level of hereditary classes, as in Question \ref{hereditary class distinctness} (and extending the generalizations suggested in Remark \ref{real-valued helix maps} to general $r'$-helix maps would imply $\mathfrak{o}(\mathcal{H})$ is not always an integer, answering both parts of that question). In the rest of this section, we discuss some additional considerations about these more general classes of helix maps, as they relate to cyclic indecomposability.

First of all, from the configuration of the linear order (i.e., the graph $\{a_{i}\}_{i < \omega}$ where $a_{i} R a_{j}$ for $i < j$), it is always possible to use directed helix maps to build any hereditarily cyclically decomposable graph:

\begin{prop}\label{directed helix maps build hereditarily cylically decompsable graphs}
    Let $n \geq 3$ be an integer, and let $\mathcal{H}$ be a hereditary class of graphs with graph relation $R$, closed under weak embeddings, which is closed under directed $m$-helix maps for integers $m \geq n$. (This is the case, for example, when $\mathcal{H}$ is $\mathrm{NSOP}_{n}$). Suppose $\{a_{i}\}_{i < \omega} \in \mathcal{H}$ where $a_{i} R a_{j}$ exactly when $i < j$. Then for any hereditarily cyclically $n$-decomposable graph $G$ (i.e., a graph $G$ none of whose induced subgraphs are cyclically $n$-indecomposable), $G \in \mathcal{H}$.
\end{prop}

\begin{proof}
    We show by induction on the size of $G$ the following: for any graph of the form $G \sqcup H$ where $G$ is hereditarily cyclically $n$-decomposable, if $G H \notin \mathcal{H}$, then there is some graph of the form $\tilde{G} \sqcup H \notin \mathcal{H}$ with $H$ as an induced subgraph, and a graph homomorphism $f: \tilde{G} \sqcup H \to G \sqcup H$ which is the identity on $H$ and maps $\tilde{G}$ to $G$, such that $\tilde{G}$ is cycle-free. This is all we need: setting $H = \emptyset$, we then see that there is some cycle-free graph not belonging to $\mathcal{H}$, contradicting $\{a_{i}\}_{i < \omega} \in \mathcal{H}$ (by Claim \ref{every cycle-free graph embeds in the linear order}).

    Suppose the hypothesis is true for all graphs smaller than $G$. Since $G$ is cyclically $n$-indecompsable, we get a full directed cyclic $n$-decomposition $G = \sqcup^{n-1}_{i =0} G_{i}$, giving a directed cyclic $n$-decomposition $G \sqcup H = \sqcup^{n-1}_{i =0} G_{i} \sqcup H$. Let $h: \tilde{G}^{*} \sqcup H \twoheadrightarrow G \sqcup H$ be the associated directed $n$-helix map of length $\ell$, where $\ell$ is large enough (by closure of $\mathcal{H}$ under directed $n$-helix maps) that $ \tilde{G}^{*} \sqcup H \notin \mathcal{H}$. Write $\tilde{G}^{*} = \bigsqcup_{j < \omega}\bigsqcup^{n-1}_{i =0} G^{j}_{i}$ as in the proof of Fact \ref{closure under helix maps}; then each of the $G^{j}_{i}$ are strictly smaller than $G$, and the only edges of $\tilde{G}^{*}$ are between two vertices of one of the $G^{j}_{i}$, or from vertices of $G^{j}_{i}$ to vertices of $G^{j'}_{i + 1 \mathrm{\: mod \:} n}$. Repeatedly applying the induction hypothesis to $G^{i}_{j}$ with its complement for each $i, j$, we find $\tilde{G} = \bigsqcup_{j < \omega}\bigsqcup^{n-1}_{i =0} (G^{j}_{i})'$ whose edges have the same properties for $(G^{j}_{i})'$, such that $\tilde{G} \sqcup H$ has a graph homomorphism to $\tilde{G}^{*} \sqcup H$ mapping $\tilde{G}$ to $\tilde{G}^{*}$ and restricting to the identity on $H$ (and therefore, has a graph homomorphism to $G \sqcup H$ mapping $\tilde{G}$ to $G$ and restricting to the identity on $H$), and additionally such that each of the $(G^{j}_{i})'$ is cycle-free. It remains only to show that $\tilde{G} = \bigsqcup_{j < \omega}\bigsqcup^{n-1}_{i =0} (G^{j}_{i})'$  is cycle-free; suppose $v_{0} \ldots, v_{k-1}$ is a cycle in $\tilde{G}$. Then this cycle cannot belong to one of the $(G^{j}_{i})'$, so for $j_{k'}$, $0 \leq k' <k$ the value of $j$ for which $v_{k'} \in \bigsqcup^{n-1}_{i =0} (G^{j}_{i})'$, where $j_{0} \leq \ldots \leq j_{k-1}$ by the properties of edges between the $(G^{j}_{i})'$, one of the inequalities must be strict. But then $j_{0} < j_{k-1}$, so it is impossible that $v_{k-1} R v_{0}$ by the properties of edges between the $(G^{j}_{i})'$, contradicting that $v_{0} \ldots, v_{k-1}$ is a cycle.
\end{proof}

Then, it may be tempting to conclude the converse: for $n \geq 3$ an integer, if a graph $G$ is contained in every hereditary class of graphs $\mathcal{H}$ that is closed under weak embeddings and is closed under directed $m$-helix maps for integers $m \geq n$ (or alternatively, is closed under all $n$-helix maps, so is $\mathrm{NSOP}_{n}$), such that $\{a_{i}\}_{i < \omega} \in \mathcal{H}$ where $a_{i} R a_{j}$ exactly when $i < j$, $G$ must be hereditarily cyclically $n$-indecomposable. If this were true, then by Lemma  \ref{(N, k)-cycle cyclically n^{*}-indecomposable}, it would give us what we want for $n = \lceil r \rceil$: the graph $\{a_{\theta}\}_{\theta \in S^{1}}$, where $a_{\theta} R a_{\psi}$ exactly when $\psi$ lies at most $\frac{2\pi}{r}$ radians counterclockwise from $\theta$, contains cyclically $n$-indecomposable subgraphs, so there must be a hereditary class of graphs $\mathcal{H}$ that is closed under weak embeddings and closed under directed $n$-helix maps for integers $m \geq n$ (or is $\mathrm{NSOP}_{n}$), such that $\{a_{i}\}_{i < \omega} \in \mathcal{H}$ but $\{a_{\theta}\}_{\theta \in S^{1}} \notin \mathcal{H}$ (so $\mathcal{H}$ has $\mathrm{SOP}_{r}$). However, this converse to Proposition \ref{directed helix maps build hereditarily cylically decompsable graphs} is false. We give a counterexample in Figure \ref{cyclically indecomposable graph with hereditarily cyclically decomposable helix cover} below, for $n = 4$. (We do not expect that there is anything special about $n=4$ in constructing this counterexample.)

\begin{figure}[hbt!]
{%
\begin{circuitikz}
\tikzstyle{every node}=[font=\small]
\draw  (5.25,6.5) circle (0.25cm) node {\small $a_3$} ;
\draw  (7,6.5) circle (0.25cm) node {\small $a_2$} ;
\draw  (8.75,6.5) circle (0.25cm) node {\small $a_1$} ;
\draw  (3.75,7.75) circle (0.25cm) node {\small $b_1$} ;
\draw  (3.75,9.25) circle (0.25cm) node {\small $b_2$} ;
\draw  (3.75,10.75) circle (0.25cm) node {\small $b_3$} ;
\draw  (5.25,12) circle (0.25cm) node {\small $c_1$} ;
\draw  (7,12) circle (0.25cm) node {\small $c_2$} ;
\draw  (8.75,12) circle (0.25cm) node {\small $c_3$} ;
\draw  (10.25,7.75) circle (0.25cm) node {\small $d_3$} ;
\draw  (10.25,9.25) circle (0.25cm) node {\small $d_2$} ;
\draw  (10.25,10.75) circle (0.25cm) node {\small $d_1$} ;
\draw  (7,3.75) circle (0.25cm) node {\small $e_1$} ;
\draw [->, >=Stealth] (7.25,4) -- (8.5,6);
\draw [->, >=Stealth] (8.25,6.5) -- (7.5,6.5);
\draw [->, >=Stealth] (6.5,6.5) -- (5.75,6.5);
\draw [->, >=Stealth] (5.5,6) -- (6.75,4);
\draw [->, >=Stealth] (6.5,4) -- (4,7.25);
\draw [->, >=Stealth] (4,10.5) -- (6.75,4);
\draw [->, >=Stealth] (10,7.5) -- (7.5,4);
\draw [->, >=Stealth] (7.25,4) -- (10,10.5);
\draw [->, >=Stealth] (6.75,4.25) -- (5.25,11.75);
\draw [->, >=Stealth] (8.75,11.75) -- (7.25,4.25);
\draw [->, >=Stealth] (5.75,12) -- (6.5,12);
\draw [->, >=Stealth] (7.5,12) -- (8.25,12);
\draw [->, >=Stealth] (3.75,8.25) -- (3.75,8.75);
\draw [->, >=Stealth] (3.75,9.75) -- (3.75,10.25);
\draw [->, >=Stealth] (10.25,10.25) -- (10.25,9.75);
\draw [->, >=Stealth] (10.25,8.75) -- (10.25,8.25);
\draw [->, >=Stealth] (10,7.75) -- (7.25,6.75);
\draw [->, >=Stealth] (5.25,6.75) -- (4,9);
\draw [->, >=Stealth] (4,11) -- (6.75,11.75);
\draw [->, >=Stealth] (9,11.75) -- (10,9.5);
\end{circuitikz}
}%
\caption{A cyclically $4$-indecomposable graph $G$ with a directed $4$-helix map $h: \tilde{G} \twoheadrightarrow G$ where $\tilde{G}$ contains no cyclically $4$-indecomposable graph.}

\label{cyclically indecomposable graph with hereditarily cyclically decomposable helix cover}
\end{figure}
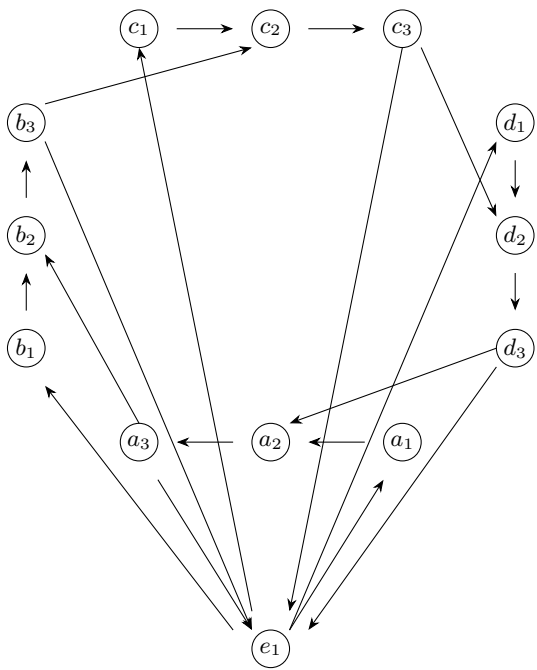

We first show that $G$ is cyclically $4$-indecomposable. Suppose $G$ has a full, directed cyclic $4$-decomposition and let $f: G \to \{0, 1, 2, 3\}$ be as in the proof of Lemma \ref{(N, k)-cycle cyclically n^{*}-indecomposable}. Also suppose, without loss of generality, that $f(e_{1}) = 0$. If $f$ takes the value $0$ on all of $G$, we are done, so by symmetry it suffices to show that $f(a_{i}) = 0$ for $i = 1, 2, 3$. Because $f(e_{1}) = 0$, $e_{1} R a_{1}$, $a_{1} R a_{2}$, $a_{2} R a_{3}$ and $a_{3} R e_{1}$, it suffices to show $f(a_{2}) = 0$. First, because $f(e_{1}) = 0$, $e_{1} R a_{1}$, and $a_{1} R a_{2}$, it is impossible that $f(a_{2}) = 3$; moreover, because $f(e_{1}) = 0$, $a_{2} R a_{3}$ and $a_{3} R e_{1}$, it is impossible that $f(a_{2}) = 1$. It remains to rule out the case that $f(a_{2}) = 2$. But because $d_{3} R e_{1}$ and $f(e_{1}) = 0$, either $f(d_{3}) = 0$ or $f(d_{3}) = 3$. In either case, because $d_{3} R a_{2}$, it is impossible that $f(a_{2}) = 2$.

Next, by Proposition \ref{directed helix maps build hereditarily cylically decompsable graphs}, to show that every hereditary class of directed graphs $\mathcal{H}$ closed under weak embeddings, closed under $m$-helix maps for integers $m \geq 4$ and with $\{a_{i}\}_{i < \omega} \in \mathcal{H}$ has $G \in \mathcal{H}$, it suffices to verify that $G$ admits a directed $4$-helix map $h: \tilde{G} \twoheadrightarrow G$ where $\tilde{G}$ contains no cyclically $4$-indecomposable induced subgraph. For $G_{0} :=\{a_{1}, a_{2}, a_{3}\} $, $G_{1} :=\{b_{1}, b_{2}, b_{3}\} $, $G_{2} :=\{c_{1}, c_{2}, c_{3}\} $, $G_{3} :=\{d_{1}, d_{2}, d_{3}\} $ and $D := \{e_{1}\}$,  $G = \bigsqcup^{3}_{i=0} G_{i} \sqcup D$ is a directed cyclic $4$-decomposition of $G$; we show the associated directed $4$-helix map  $h: \tilde{G} \twoheadrightarrow G$ is as desired. Write $\tilde{G} = \bigsqcup_{j < \omega}\bigsqcup^{3}_{i =0} G^{j}_{i}$ as in the proof of Fact \ref{closure under helix maps}; we show that, for finite $\tilde{G}_{0} \subset \tilde{G}$ of size at least $2$, $\tilde{G}_{0}$ is cyclically $4$-decomposable. Let $j < \omega$ be the least such that $\tilde{G}_{0}$ contains points of $\bigsqcup^{3}_{i =0} G^{j}_{i}$; we may assume, without loss of generality, that $\tilde{G}_{0}$ contains points of $G^{j}_{0}$. Let $\tilde{a_{1}}, \tilde{a_{2}}, \tilde{a_{3}} \in G^{j}_{0}$ map to $a_{1}, a_{2}, a_{3}$ respectively; then $\tilde{G}_{0}$ contains at least one of $\tilde{a_{1}}, \tilde{a_{2}}, \tilde{a_{3}}$. If it does not contain all three of them as well as $e_{1}$, then one of the $\tilde{a_{1}}, \tilde{a_{2}}, \tilde{a_{3}}$ that does belong to $\tilde{G}_{0}$ will either have no edges to other vertices of $\tilde{G}_{0}$, or no edges from other vertices of $\tilde{G}_{0}$, yielding a full, directed cyclic $4$-decomposition of $\tilde{G}_{0}$ as in the proof of Lemma \ref{cyclically n-indecomposable graphs have cycles}. Otherwise, all of $\tilde{a_{1}}, \tilde{a_{2}}, \tilde{a_{3}}$ as well as $e_{1}$ belong to $\tilde{G}_{0}$; then $f(\tilde{a}_{i}) = i$ for $i = 1, 2, 3$ and $f(v) = 0$ for $v \in \tilde{G}_{0} \backslash \{\tilde{a_{1}}, \tilde{a_{2}}, \tilde{a_{3}}\}$ gives a full, directed cyclic $4$-decomposition of $\tilde{G}_{0}$. (The only edges of $\tilde{G}_{0}$ one of whose endpoints is $\tilde{a_{1}}$ are an edge from $e_{1} \in \tilde{G}_{0} \backslash \{\tilde{a_{1}}, \tilde{a_{2}}, \tilde{a_{3}}\}$ and an edge to $\tilde{a_{2}}$, the only edges of $\tilde{G}_{0}$ one of whose endpoints is $\tilde{a_{3}}$ are an edge from $\tilde{a_{2}}$ and edges to vertices $v \in \tilde{G}_{0} \backslash \{\tilde{a_{1}}, \tilde{a_{2}}, \tilde{a_{3}}\}$, and since $j < \omega$ is the least such that $\tilde{G}_{0}$ contains points of $\bigsqcup^{3}_{i =0} G^{j}_{i}$, the only edges of $\tilde{G}_{0}$ one of whose endpoints is $\tilde{a_{2}}$ are an edge to $\tilde{a}_{3}$ and an edge from $\tilde{a_{1}}$.) So $\tilde{G}_{0}$ is in fact cyclically $4$-decomposable.

We summarize the discussion of this counterexample in the following proposition, which says, informally, that the combinatorial strength of $\mathrm{NSOP}_{n}$ surpasses just being able to build (from $\{a_{i}\}_{i <\omega}$ with $a_{i} R a_{j}$ exactly when $a_{i} < a_{j}$) the hereditarily cyclically $n$-indecomposable graphs.

\begin{prop}
There is an integer $n > 3$ and a cyclically $n$-indecomposable graph $G$ such that, for every hereditary class of graphs $\mathcal{H}$ with graph relation $R$ closed under weak embeddings, with $\mathcal{H}$ closed under directed $m$-helix maps for integers $m \geq n$ (so in particular, for when $\mathcal{H}$ is $\mathrm{NSOP}_{n}$), if the graph $\{a_{i}\}_{i < \omega} \in \mathcal{H}$ where $a_{i} R a_{j}$ exactly when $i < j$, then $G \in \mathcal{H}$.
\end{prop}

This means that, in order to extend Theorem \ref{interval helix non-integrality} to show that the real-valued $\mathrm{NSOP}_{r}$ hierarchy is distinct from the integer-valued $\mathrm{NSOP}_{n}$ hierarchy for hereditary classes or that $\mathfrak{o}(\mathcal{H})$ is not always an integer, cyclic indecomposability alone will not be enough. However, it also means that, if we are able to show distinctness of the real-valued and integer valued hierarchies for hereditary classes or that $\mathfrak{o}(\mathcal{H})$ has non-integer values, our results on interval helices in Theorem \ref{interval helix non-integrality} will still provide us with further information, giving a combinatorial invariant preserved by interval helices that is not preserved by helix maps in general.

\textbf{Acknowledgements} The author would like to thank Maryanthe Malliaris for many wise words during the process of writing this article. We also would like to thank Sylvy Anscombe for many engaging discussions, and for asking questions that clarified some of this article's results, as well as Caroline Terry for bringing the connection to \cite{LZ17} in the proof of Theorem \ref{integerality} to the author's knowledge.

\bibliographystyle{plain}
\bibliography{refs}

\end{document}